\documentclass[a4paper,11pt]{amsart}

\usepackage{amssymb,amscd,amsmath}
\usepackage[mathcal]{eucal}
\usepackage{array,float}

\usepackage{xcolor}

\usepackage{xy}
\input xy
\xyoption{all}
\usepackage{pdflscape}
\usepackage{hyperref}
\usepackage[left=3cm, right=3cm]{geometry}
\usepackage{enumerate}
\numberwithin{equation}{section}
\numberwithin{equation}{subsection}

\newtheorem{thm}{Theorem}[section]

\newtheorem{corollary}[thm]{Corollary}
\newtheorem{lemma}[thm]{Lemma}

\newtheorem{definition}[thm]{Definition}
\newtheorem*{remark*}{Remark}
\newtheorem{remark}[thm]{Remark}

\newcommand{\Hom}{{\mathrm{Hom}}}

\newcommand{\tra}{{\mathrm{tra}}}
\newcommand{\res}{{\mathrm{res}}}

\newcommand{\Ho}{{\mathrm{H}}}

\newcommand{\vt}{{\bar{\alpha}}^{\prime}}

\newcommand{\irr}{\mathrm{Irr}}

\title[On the Twisted group ring isomorphism  problem]{On the Twisted group ring isomorphism  problem for a class of groups }

\author{Sumana Hatui}
\address{School of Mathematical Sciences, National Institute of Science Education and Research, An OCC of Homi Bhabha National Institute, Bhubaneswar 752050, Odisha, India}

\email{sumanahatui@niser.ac.in}

\author{Gurleen Kaur}
\address{Department of Mathematics, Indian Institute of Technology Ropar, Punjab 140001,India}

\email{gurleenkaur992gk@gmail.com}

\author{Sahanawaj Sabnam}
\address{School of Mathematical Sciences, National Institute of Science Education and Research, An OCC of Homi Bhabha National Institute, Bhubaneswar 752050, Odisha, India}

\email{sahanawaj.sabnam@niser.ac.in}

\begin{document}
\date{}
\begin{abstract}
 The twisted group ring isomorphism problem (TGRIP) is a variation of the classical group ring isomorphism problem. It asks whether the ring structure of the twisted group ring determines the group up to isomorphism. 
In this article, we study the TGRIP for direct product and central product of groups. 
We provide some criteria to answer the TGRIP for groups by answering the TGRIP for the associated quotients. As an application of these results, we provide several examples. 
Finally, we answer the TGRIP for extra-special $p$-groups, and for the groups of order $p^5$, where $p \geq 5$ is a prime, except a list of five groups. 
\end{abstract}
\maketitle

\section{Introduction}
The group ring $RG$ of a finite group $G$ over a commutative ring $R$ is an important area of study in representation theory. The classical group ring isomorphism problem asks whether the $R$-ring structure of $RG$ determines the group $G$ up to isomorphism. Specifically, if $RG$ and $RH$ are isomorphic as $R$-rings, does it imply that the groups $G$ and $H$ are also isomorphic? An obvious case of non-isomorphic groups having isomorphic group rings is that of two abelian groups of the same order over any field containing roots of unity of that order \cite{PW}. In 1971, Dade constructed an example of two non-isomorphic groups  whose group algebras are isomorphic over any field. In both examples, the corresponding integral group rings are not isomorphic. In \cite{Hertweck}, Hertweck provided an example of two non-isomorphic groups of even order whose integral group rings are isomorphic. However, the (integral) group ring isomorphism problem is still open for groups of odd order. The classical group ring isomorphism problem has been studied by several authors, most notably in 
\cite{MR142622, mr602901,margolis2018finite, MR812486}.

  We denote the set of $2$-cocycles of $G$ over the unit group $R^\times$ by $Z^2(G, R^\times)$ and the second cohomology group of $G$  by $\Ho^2(G, R^\times)$.
 For a $2$-cocycle $\alpha$ of $G$, $[\alpha]$ denotes the corresponding cohomology class in $\Ho^2(G, R^\times)$, where $R^\times$ is a trivial $G$-module. 
  In \cite{MS}, Margolis and Schnabel introduced a twisted analog of the classical  group ring isomorphism problem, namely the twisted group ring isomorphism problem. 
 For a group $G$, and for $\alpha\in Z^2(G, R^\times)$,  the twisted group ring $R^\alpha G$ of group $G$ over the commutative ring $R$ is $R^\alpha G = \{\sum_{g \in G} x_g \bar{g} \mid x_g \in R \}$,  a free $R$-module with the basis $\{\bar{g}, g \in G\}$ and with the multiplication on the basis elements  given by $\bar{g} \bar{h}=\alpha(g,h)\overline{gh}$
(see \cite[Section 2, Page 77]{Karpilovsky} for further details). 
 Twisted group rings play an important role in the study of projective representation of groups, which was shown by Schur in his pioneering works \cite{IS4, IS7, IS11}. 
 
\begin{definition}
    We say that $G  \sim_R H$ via the map $\phi$ if there is an isomorphism  $\phi: \Ho^2(G, R^\times)  \to  \Ho^2(H, R^\times)$ such  that $R^\alpha G\cong  R^{\phi(\alpha)} H$ for each  $[\alpha] \in  \Ho^2(G, R^\times)$.
\end{definition}
The twisted group ring isomorphism problem (TGRIP) asks to determine the equivalence classes of groups of order $n$ under the equivalence relation `$ \sim_R$'. 

In \cite{MS}, an answer to the TGRIP over $\mathbb C$ has been provided for finite abelian groups, some groups of central type, and groups of cardinality $p^4$ and $p^2q^2$ for $p,q$ primes. In \cite{PGS}, the TGRIP has been studied for finite non-abelian $p$-groups $G$ by fixing the order of the Schur multiplier $\Ho^2(G, \mathbb C^\times)$.
Also, the TGRIP over different fields has been investigated in \cite{Jespers1,MR4145799,MR4512533}. 

In this article, the TGRIP has been investigated for finite groups $G$ over  $\mathbb{C}$.
We provide a criterion to solve the TGRIP for  groups by solving it for the quotient groups. 
We  discuss the TGRIP for  direct product  and central product of groups. 
First, we prove that to study the TGRIP for finite nilpotent groups, it is enough to study it for $p$-groups.
For the direct product of groups, we propose several conditions to answer the TGRIP of the group by answering for its components, which  help to construct many examples of groups.
  Also, we give a criterion for studying the TGRIP of the central product of groups. As an application,  we discuss some examples and provide an answer for extra-special $p$-groups. Finally, we investigate the TGRIP for the groups of order $p^5$, where $p \geq 5$ is a prime, except a list of few groups. 



Before stating our main results, we define some notation.
The subgroups $G'$ and $Z(G)$ denote the derived subgroup and center of a group $G$. 
For $x,y \in G,$ the commutator $x^{-1}y^{-1}xy$ is denoted by $[x,y]$. We use $G^p$ to denote the subgroup of $G$ generated by $p$-th power of the elements of $G$ and
$C_{p^n}$ denotes the cyclic group of order $p^n$.
Now we recall three homomorphisms: restriction, inflation, and transgression.  Let $H \leq G$ and $\alpha \in Z^{2}(G,\mathbb{C}^{\times})$. Then the restriction $\alpha'=\alpha|_{H \times H}\in Z^{2}(H,\mathbb{C}^{\times})$ and hence induces a homomorphism 
$\res^G_H: \Ho^{2}(G,\mathbb{C}^{\times}) \to \Ho^{2}(H,\mathbb{C}^{\times})$ by $\res^G_H([\alpha])=[\alpha']$, known as the restriction map. 
Suppose $H \triangleleft$ $G$. For a given $\alpha \in Z^{2}(G/H,\mathbb{C}^{\times})$, define $\alpha': G \times G \to \mathbb C^\times$ by $\alpha'(x,y)=\alpha(xH,yH)$. Then the assignment $[\alpha] \mapsto [\alpha']$ induces a homomorphism $\inf: \Ho^{2}(G/H,\mathbb{C}^{\times}) \to \Ho^{2}(G,\mathbb{C}^{\times})$, called the inflation map. 
Let $1 \to H \to G^{*} \xrightarrow{f} G \to 1$ be a central  extension and  $\mu$ be a section of $f$. For any $\chi \in \operatorname{Hom}(H,\mathbb{C}^{\times})$, define $\alpha_{\chi}: G \times G \to \mathbb{C}^{\times}$ by $\alpha_{\chi}(x,y)=\chi(\mu(x)\mu(y)\mu(xy)^{-1})$. Then the assignment $\chi \mapsto [\alpha_{\chi}]$ induces a homomorphism $\tra: \operatorname{Hom}(H,\mathbb{C}^{\times}) \to \Ho^{2}(G,\mathbb{C}^{\times})$, called the transgression map associated with the given central extension.

Our first main result is the following, which shows that under certain conditions, the TGRIP for certain quotient group answers that for the group.
Here, $\tra_j$ and $\inf_j$, $j=1,2$, denote the corresponding transgression and inflation homomorphisms.
\begin{thm}\label{mainthm}
	Let $G_1, G_2$ be two groups and  $Z_j$ be  central subgroups of $G_j$, for $j=1,2$, such that the sequences 
	\[
	1  \xrightarrow{}   \mathrm{Hom}(Z_{j}, \mathbb C^\times)  \xrightarrow{\mathrm{tra}_j}  \Ho^2(G_j/Z_{j}, \mathbb C^\times)   \xrightarrow{\mathrm{inf}_j}  \Ho^2(G_j, \mathbb C^\times)  \xrightarrow{}  1 
	\]
	are exact.
	Suppose there is an isomorphism $\bar{i}:  \mathrm{Hom}(Z_{1}, \mathbb C^\times) \to  \mathrm{Hom}(Z_{2}, \mathbb C^\times)$ and  $G_1/Z_1 \sim_\mathbb C G_2/Z_2$ via the map $\bar{\phi}$ such that the following diagram is commutative.	
	\begin{figure}[H]
	\[
	\xymatrix{ 
		1 \ar[r] &  \mathrm{Hom}(Z_{1}, \mathbb C^\times) \ar[d]^{\bar{i}} \ar[r] ^{\mathrm{tra}_{1}} & \Ho^2(G_1/Z_{1}, \mathbb C^\times)  \ar[d]^{\bar{\phi}} \\
		1 \ar[r] &  \mathrm{Hom}(Z_{2}, \mathbb C^\times) \ar[r] ^{\mathrm{tra}_{2}} & \Ho^2(G_2/Z_{2}, \mathbb C^\times)\\
	}
	\] \caption{Diagram 1}
	\end{figure}
	
	Then $G_1 \sim_\mathbb C G_2$.
	\end{thm}

 We say that an isomorphism $\phi: G_2 \to G_1$ induced an isomorphism $\bar{\phi}: \Ho^2(G_1, \mathbb C^\times) \to  \Ho^2(G_2, \mathbb C^\times)$ if $\bar{\phi}$ is defined by $\bar{\phi}([\alpha])=[\beta]$ such that $\beta (g, h)=\alpha(\phi(g),\phi(h))$ for $g, h \in G_2$.
Now, the  next corollary   follows immediately from the above result.
 \begin{corollary}\label{result1}
  Let $G_1, G_2$ be two groups and  $Z_j$ be central subgroups of $G_j$ such that for $j=1,2$, the sequences 
 \[
 1  \xrightarrow{}   \mathrm{Hom}(Z_{j}, \mathbb C^\times)  \xrightarrow{\mathrm{tra_j}}  \Ho^2(G_j/Z_{j}, \mathbb C^\times)   \xrightarrow{\mathrm{inf}_j}  \Ho^2(G_j, \mathbb C^\times)  \xrightarrow{}  1 
\]
 are exact. Suppose
 there are isomorphisms $i: Z_2\to Z_1$ and $\phi: G_2/Z_2 \to G_1/Z_1$ such that Diagram $1$ is commutative for the induced isomorphisms $\bar{i}$ and $\bar{\phi}$.
	Then $G_1 \sim_\mathbb C G_2$.
 \end{corollary}
 
\noindent  A proof of Theorem \ref{mainthm} is given in Section \ref{Mainresult}. As an application of these results, in Theorem \ref{application1}, we  study the TGRIP for some groups of order $p^6$, where $p$ is an odd prime. We also use Corollary \ref{result1}  to solve the TGRIP for some non-abelian groups of order $p^5$, where $p$ is an odd prime (see proof of $(xvi)$ of Theorem \ref{groupsorderp5}).

In Section \ref{section1}, we study the TGRIP for the direct product of groups. Since every finite nilpotent group is isomorphic to the  direct product of its Sylow $p$-subgroups,  the following result tells that to study the TGRIP for finite nilpotent groups, it is enough to study the TGRIP  for $p$-groups. 

\begin{thm}\label{direct_product}Let $G=G_{1} \times G_{2} \times \cdots \times G_{r}$ and $H=H_{1}\times H_{2} \times \cdots \times H_{r}$ be two groups such that $|G_{i}|=|H_{i}|=k_{i}$ and $(k_{i},k_{j})=1$, if $i \neq j$. Then $G \sim_{\mathbb{C}} H$ if and only if $G_{i} \sim_{\mathbb{C}} H_{i}$ for $i=1,2,\cdots,r$.
\end{thm}

The next result enables us to  construct  many examples of direct product of groups which are related w.r.t the relation $\sim_{\mathbb C}$.

\begin{corollary}\label{coro1}
	Suppose  $G_i, H_i$ are two groups  satisfying the hypotheses of Corollary \ref{result1} for $i=1,2$. 
	Then the following hold: 
	\begin{enumerate}[(i)]
		\item $G_1\times  G_2\sim_\mathbb C H_1\times H_2.$
		
		\item For every finite abelian group $A$, $G_1\times  A\sim_\mathbb C H_1\times A.$
	\end{enumerate}
\end{corollary} 

In Section \ref{section2}, we study the TGRIP for the central product  of groups.
A group $G$ is a central product of its normal subgroups $G_1$ and $G_2$ amalgamating $A$ if $G=G_1G_2$ with $A=G_1 \cap G_2$ and $[G_1 ,G_2 ]=1$.

\begin{thm}\label{central product}Let $G$ be a central product of its normal subgroups $G_{1}$ and $G_{2}$, and $H$ be a central product of its normal subgroups $H_{1}$ and $H_{2}$. Denote $G_{1}'\cap G_{2}'$ by $Z_{1}$ and $H_{1}'\cap H_{2}'$ by $Z_{2}$. 
Suppose   $G /Z_{1} \sim_\mathbb{C} H/Z_{2}$ via the map $\bar{\phi}$, and
there is an isomorphism  $\bar{i}:   \mathrm{Hom}(Z_{1}, \mathbb C^\times)\to   \mathrm{Hom}(Z_{2}, \mathbb C^\times)$  such that Diagram $1$ is commutative.
Then
 $G \sim_\mathbb{C} H.$
  \end{thm} 
 As an application, we provide examples in Theorem \ref{example2} and
 we also answer the TGRIP for extra-special $p$-groups, which is given in the following result.
\begin{thm}\label{extra_special}
	Let $G$ and $H$ be two extra-special $p$-groups of order $p^{2n+1}, n\geq 1$. Then the following hold:
	\begin{enumerate}[(i)]
		\item If $n=1$, then the equivalence class is a singleton set w.r.t `$\sim_{\mathbb{C}}$'.
		\item If $n>1$, then $G \sim_{\mathbb{C}} H$.
	\end{enumerate}
\end{thm}

The groups of order $p^5$ and $p^6$ for odd prime $p$, had been classified in \cite{James}, and some errors of this paper have been identified in \cite{correction}.
There are $10$ isoclinism classes for groups of order $p^5$, where $p$ is an odd prime, denoted by $\Phi_i, 1\leq i \leq 10$.
 The isoclinism class $\Phi_1$ consists of abelian groups.
  We suggest readers to keep the article \cite{James} and \cite{correction} handy as we use the classification and notations of the groups given in \cite{James} throughout this paper, without further reference. Note that when we write a presentation of a group, we omit the relations of the form $[x,y]=1$, for the generators $x,y \in G$.
  In Section \ref{groupsp5}, we answer the TGRIP for the groups of order $p^5, p\geq 5$, except the groups given in the list $S=\{\Phi_3(221)b_r,\Phi_4(221)b,   \Phi_4(221)d_{\frac{1}{2}(p-1)},  \Phi_4(221)f_0,
 \Phi_{6}(1^5)\}$.

 \section{Preliminaries}
  In this section, we recall some results used to prove our main results.
For  $\alpha \in Z^2(G,  \mathbb{C}^\times)$,  the projective representations of $G$ corresponding to $\alpha$ are called  $\alpha$-representations of $G$. See \cite[page  71]{Karpilovsky}  for further details.
For a finite group $G$, by Wedderburn decomposition [\cite{MR1183469}, Theorem 2.1.2], we have  
 $$\mathbb{C}^\alpha G \cong \prod_{i=1}^t \mathbb{C}^{n_i \times n_i},$$ 
 where $t$ is the number of inequivalent irreducible $\alpha$-representations of $G$, and $n_{i}$ is the dimension of the $i$-th irreducible  $\alpha$-representation of $G$, and $\mathbb{C}^{n_i \times n_i}$ denotes $n_i \times n_i$ matrices with entries in $\mathbb C$. 
  For a finite group $G$, the set of all inequivalent irreducible complex $\alpha$-representations of $G$ will be denoted by $\irr^\alpha(G)$, and $\irr(G)$ denotes the set of all inequivalent irreducible ordinary representations of $G$.
A group $G$ is capable if there is a group $H$ such that $G \cong H/Z(H)$.  The epicenter of $G$ is denoted by $Z^*(G)$ which is the smallest central subgroup
of $G$ such that $G/Z^*(G)$ is capable.

\begin{thm}[Theorem 2.5.10, \cite{GK}]\label{capable}
If $G$ is a finite group with epicenter $Z^*(G)$, then for any subgroup $Z \subseteq Z^*(G)\cap G'$ of $G$,  the following  sequence 
\[
 1  \xrightarrow{}   \mathrm{Hom}(Z, \mathbb C^\times)  \xrightarrow{\mathrm{tra}}  \Ho^2(G/Z, \mathbb C^\times)   \xrightarrow{\mathrm{inf}}  \Ho^2(G, \mathbb C^\times)  \xrightarrow{}  1 
\]
is exact.
\end{thm}

\begin{thm}[Theorem 2.2.10, \cite{GK}]\label{kar} 
For two groups $G_1$ and $G_2$, 
$$\Ho^2(G_1\times G_2, \mathbb C^\times)\cong \Ho^2(G_1, \mathbb C^\times) \times \Ho^2(G_2, \mathbb C^\times)\times \Hom(G_1/G_1' \otimes_\mathbb Z G_2/G_2', \mathbb C^\times).$$
\end{thm}
The following result follows from \cite[Theorem 3.2]{PS}.
\begin{thm}\label{dimpreserve}
Let $G$ be a finite group and $Z\subseteq G' \cap Z(G)$ be a subgroup of $G$. 
If $[\alpha] \in \mathrm{Im}\big(\inf:\Ho^2(G/Z, \mathbb C^\times) \to \Ho^2(G, \mathbb C^\times)\big)$, then there is a dimension preserving bijective correspondence between the sets $\irr^\alpha(G)$
and $\bigcup_{\{[\beta] \in\Ho^2(G/Z, \mathbb C^\times) \mid \inf([\beta])=[\alpha]\}}\irr^\beta(G/Z)$. In particular, 
	\[
		\mathbb C^\alpha G \cong \prod_{\{[\beta]  \in\Ho^2(G/Z, \mathbb C^\times) \mid \inf([\beta])=[\alpha]\}} \mathbb C^\beta [G/Z].
		\]
\end{thm}

\begin{definition}[\cite{GK}, page 16]\label{def}
A group $\tilde{G}$ is a representation group (or covering group) of $G$ if there is a central extension 
\[
1 \to Z \to \tilde{G} \to G \to 1
\]
such that $Z \subseteq Z(\tilde{G})\cap (\tilde{G})'$ and $Z \cong \Ho^2(G,\mathbb C^\times)$.
\end{definition}
If $\tilde{G}$ is a representation group of $G$, then it follows from \cite[Theorem 2.1.4]{GK} that the corresponding transgression map $\mathrm{tra}:  \Hom(Z, \mathbb C^\times) \to \Ho^2(G, \mathbb C^\times)$ is an isomorphism.
Hence the following result follows from  \cite[Theorem 3.2, Corollary 3.3]{PS}.

\begin{thm}\label{proj-ord}
Let $G$ be a finite group and $\tilde{G}$ be a representation group of $G$ with a central subgroup $Z$ such that $\tilde{G}/Z \cong G$, as in Definition \ref{def}.
Let $\chi \in \Hom(Z, \mathbb C^\times)$ such that $\tra(\chi)=[\alpha] \in \Ho^2(G, \mathbb C^\times)$. Then there is a bijective correspondence between the sets $\irr^\alpha(G)$ and $\irr(\tilde{G}\mid \chi)$, where $\irr(\tilde{G}\mid \chi)=\{\rho \in \irr(\tilde{G}) \mid \langle \rho|_Z, \chi\rangle\neq 0\}$ denotes the irreducible complex ordinary representations of $\tilde{G}$ lying above $\chi$. 
In particular, the sets 
$\cup_{[\alpha] \in \Ho^2(G, \mathbb C^\times) }\irr^\alpha(G)$ and $\irr(\tilde{G})$ are in bijective correspondence.
\end{thm}

\begin{lemma}[Lemma 4.2 of \cite{PGS}]\label{bijection_ord}
Let $G_1$ and $G_2$ be two finite groups with $\tilde{G}_1$ and $\tilde{G}_2$ as their representation groups, respectively. For $i \in \{1, 2\}$, let $Z_i$ be a central subgroup of $\tilde{G}_i$ such that $\tilde{G}_i /Z_i\cong G_i$ and the transgression maps $tra_i : \Hom(Z_i, \mathbb C^\times) \to \Ho^2(G_i, \mathbb C^\times)$ are isomorphisms. Let $\sigma : \Hom(Z_1, \mathbb C^\times) \to \Hom(Z_2, \mathbb C^\times)$ be an isomorphism such that for every $\chi \in  \Hom(Z_1, \mathbb C^\times)$,  there is a dimension preserving bijection between the sets $\irr(\tilde{G}_1 \mid \chi)$ and $ \irr(\tilde{G}_2 \mid \sigma(\chi))$.
Then $G_1 \sim_\mathbb C G_2$.
\end{lemma}

 \section{Main results}\label{Mainresult}
%
%

\noindent
\textbf{Proof of Theorem \ref{mainthm}.} Let $[\alpha]\in \Ho^{2}(G_1,\mathbb{C}^{\times}) $. Then, there is a $[\beta] \in  \Ho^{2}(G_1/Z_1,\mathbb{C}^{\times}) $ such that $\inf_1([\beta])=[\alpha]$.
	Define ${\psi} ~:~\Ho^{2}(G_1,\mathbb{C}^{\times}) \to \Ho^{2}(G_2,\mathbb{C}^{\times})$ by 
	$${\psi} ([\alpha]) = {\inf}_2\circ \bar{\phi}([\beta]) .$$
	 First, we prove that the map $\psi$ is well defined. Let $[\alpha]=\inf_{1}([\beta_{1}]) =\inf_{1}([\beta_{2}]),$ for $[\beta_{j}]\in  \Ho^2(G_1/Z_{1}, \mathbb C^\times)$, $j=1,2$. This implies that $[\beta_{2}][\beta_{1}]^{-1}= \operatorname{tra}_{1}(\chi)$ for some $\chi \in \operatorname{Hom}(Z_{1},\mathbb{C}^{\times})$. Hence $[\beta_{2}]=\operatorname{tra}_{1}(\chi).[\beta_{1}]$. Consequently,
	 $${\inf}_{2} \circ \bar{\phi}([\beta_{2}])=\operatorname{inf}_{2}(\bar{\phi}[\beta_{1}])\operatorname{inf}_{2}(\bar{\phi}\circ \operatorname{tra}_{1}(\chi))=\operatorname{inf}_{2}(\bar{\phi}[\beta_{1}]).$$ This proves that the map $\psi$ is well defined. 
	 It is easy to see that ${\psi} $ is a surjective homomorphism as $\inf_j$ are surjective maps, for $j=1,2$. 
	 
	 Next, we prove that ${\psi}$ is injective. For this, let ${\psi} ([\alpha_{1}])={\psi} ([\alpha_{2}])$ for $[\alpha_{j}]\in  \Ho^2(G_1, \mathbb C^\times)$, where
	 $[\alpha_{j}]=\operatorname{inf}_{1}([\beta_{j}])$ for some $[\beta_{j}]\in  \Ho^2(G_1/Z_{1}, \mathbb C^\times)$ for $j=1,2$. Then $\operatorname{inf}_{2}\circ \bar{\phi}([\beta_{2}][\beta_{1}]^{-1}) =1$. 
	 Consequently, $\bar{\phi}([\beta_{2}][\beta_{1}]^{-1})=\operatorname{tra}_{2}(\chi_{2}) =\operatorname{tra}_{2}(\bar{i}(\chi_{1}))=\bar{\phi} \circ \operatorname{tra}_{1}(\chi_{1})$  for some $\chi_j \in \mathrm{Hom}(Z_{j}, \mathbb C^\times)$. This implies that $[\beta_{2}] = [\beta_{1}].\operatorname{tra}_{1}(\chi_{1})$. 
	 Hence $[\alpha_{2}]=\operatorname{inf}_{1}([\beta_{1}]).\operatorname{inf}_{1}(\operatorname{tra}_{1}(\chi_{1})) = [\alpha_{1}].$
	 Thus ${\psi}$ is an isomorphism.
Observe that the following diagram is commutative.	 \[
\xymatrix{ 
	1 \ar[r] &  \mathrm{Hom}(Z_{1}, \mathbb C^\times) \ar[d]^{\bar{i}} \ar[r] ^{\mathrm{tra}_{1}} & \Ho^2(G_1/Z_{1}, \mathbb C^\times)  \ar[d]^{\bar{\phi}}  \ar[r] ^{\mathrm{inf}_{1}} & \Ho^2(G_1, \mathbb C^\times)  \ar[d]^{{\psi}}  \ar[r] & 1\\
	1 \ar[r] &  \mathrm{Hom}(Z_{2}, \mathbb C^\times) \ar[r] ^{\mathrm{tra}_{2}} & \Ho^2(G_2/Z_{2}, \mathbb C^\times) \ar[r] ^{\mathrm{inf}_{2}} & \Ho^2(G_2, \mathbb C^\times) \ar[r] & 1.\\
}
\]		 
Due to Theorem \ref{dimpreserve} and by the assumptions, for every  $[\alpha] \in \Ho^{2}(G_1,\mathbb{C}^\times) $, we have
	\begin{eqnarray*}
		\mathbb C^\alpha G_1   && \cong \prod_{\{[\beta] \in   \Ho^2(G_1/Z_{1}, \mathbb C^\times) \mid \inf_1([\beta]) = [\alpha]\}} \mathbb C^\beta [G_1/Z_1] 
		 \cong 
				 \prod_{\{[\beta]\in   \Ho^2(G_1/Z_{1}, \mathbb C^\times) \mid \inf_1([\beta])=[\alpha]\}} \mathbb C^{\bar{\phi}(\beta)} [G_2/Z_2] \\
				 && \cong
				 \prod_{\{[\beta'] \in  \Ho^2(G_2/Z_{2}, \mathbb C^\times) \mid \inf_2([\beta'])={\psi}([\alpha]\}} \mathbb C^{\beta'} [G_2/Z_2]\\
&& \cong \mathbb C^{{\psi}(\alpha)} G_2.
	\end{eqnarray*}
Therefore, $\mathbb C^\alpha G_1 \cong \mathbb C^{{\psi}(\alpha)} G_2$ for each $[\alpha] \in \Ho^{2}(G_1,\mathbb{C}^\times) $.
Hence $G_1\sim_{\mathbb C} G_2$ via the map ${\psi}$. 

\qed


%
%
%

\begin{remark}
Lemma $3.2$ of  \cite{PGS} follows from Corollary $\ref{result1}$ by taking $Z_i=G_i'$. 
 \end{remark}

 As an application of Corollary  \ref{result1} we have the following result. Note that in the following theorem  $\Phi_{12}, \Phi_{13}$ are two isoclinism classes of groups of order $p^6$ given in \cite{James}.
 
 \begin{thm}\label{application1}
 The set $\{\Phi_i(21^4)a, \Phi_i(21^4)b, \Phi_i(21^4)c, \Phi_i(21^4)d, \Phi_i(21^4)e; ~ i=12, 13\}$ 
 form the following equivalence classes w.r.t the relation $\sim_\mathbb C~:$
$\{\Phi_{12}(21^4)a, \Phi_{12}(21^4)c, \Phi_{12}(21^4)d\},\\
\{ \Phi_{12}(21^4)b, \Phi_{12}(21^4)e\}$, 
$\{\Phi_{13}(21^4)a, \Phi_{13}(21^4)b, \Phi_{13}(21^4)c,  \Phi_{13}(21^4)d, \Phi_{13}(21^4)e\}$.
\end{thm}
\begin{proof}
By \cite[Theorem 2.2]{Bioch}, the complex group algebras 
$\mathbb C G$ are isomorphic for all groups $G$ belonging to the same isoclinism class.
Hence the complex group algebras are isomorphic for all groups in $\Phi_{12}$ and for all groups in $\Phi_{13}$. By \cite[Table 4.1]{James}, if $G \in \Phi_{12}$, $H  \in \Phi_{13}$, then $\mathbb C G \not\cong \mathbb C H$.

All these groups $G$ are special $p$-groups of rank $2$ with $G/G'\cong (C_p)^4$ and $G^p\cong C_p$. 
From \cite[Theorem 1.3]{special} and Theorem \ref{capable} it follows that the sequence  \[
 1  \xrightarrow{}   \mathrm{Hom}(G^p, \mathbb C^\times)  \xrightarrow{\mathrm{tra}}  \Ho^2(G/G^p, \mathbb C^\times)   \xrightarrow{\mathrm{inf}}  \Ho^2(G, \mathbb C^\times)  \xrightarrow{}  1
\]
is exact, and thus $\Ho^2(G, \mathbb C^\times)$ can be computed. Since  $\Ho^2(\Phi_{12}(21^4)a, \mathbb C^\times)\ncong  \Ho^2(\Phi_{12}(21^4)b, \mathbb C^\times)$, we have $\Phi_{12}(21^4)a \nsim_\mathbb C \Phi_{12}(21^4)b$.

Consider 
$$G_1=\Phi_{12}(21^4)b=\langle \alpha_i, \beta_i, \gamma_i; i=1,2\mid [\alpha_i, \beta_i]=\gamma_i, \alpha_1^p=\gamma_1\gamma_2, \alpha_2^p=\beta_i^p=\gamma_i^p=1\rangle$$
and 
$$G_2=\Phi_{12}(21^4)e=\langle \alpha^\prime_i, \beta^\prime_i, \gamma^\prime_i; i=1,2\mid [\alpha^\prime_i, \beta^\prime_i]=\gamma^\prime_i, {\alpha_1^\prime}^p={\alpha_2^\prime}^p=\gamma^\prime_1\gamma^\prime_2, {\beta_i^\prime}^p={\gamma_i^\prime}^p=1\rangle.$$ 
Now, taking $Z_{1} =G_1^p$, $Z_{2} =G_2^p,$
 we see that $Z_{1} \cong Z_{2} \cong C_{p}$ and 
 $$G_1/Z_{1} \cong G_2/Z_{2}\cong  \langle \bar{\alpha}_i,\bar{\beta}_i;  i=1,2\mid [\bar{\alpha}_1,\bar{\beta}_1]=[\bar{\beta}_2,\bar{\alpha}_2]=\gamma; \bar{\alpha}_i^p=\bar{\beta}_i^p=\gamma^p=1 \rangle.$$
Hence $ \mathrm{Hom}(Z_{1}, \mathbb C^\times) \cong   \mathrm{Hom}(Z_{2}, \mathbb C^\times)$ and  $\Ho^{2}(G_1/Z_{1},\mathbb{C}^{\times}) \cong \Ho^{2}(G_2/Z_{2},\mathbb{C}^{\times}).$  
Consider the following diagram
	 \[
	\xymatrix{ 
		1 \ar[r] &  \mathrm{Hom}(Z_{1}, \mathbb C^\times) \ar[d]^{\bar{i}} \ar[r] ^{\mathrm{tra}_{1}} & \Ho^2(G_1/Z_{1}, \mathbb C^\times)  \ar[d]^{\bar{\phi}} \\
		1 \ar[r] &  \mathrm{Hom}(Z_{2}, \mathbb C^\times) \ar[r] ^{\mathrm{tra}_{2}} & \Ho^2(G_2/Z_{2}, \mathbb C^\times),\\
	}
	\]
where the transgression maps  are defined as follows:
	define a section $s_{1}: G_1/Z_{1} \to G_1$ by 
	$s_{1}(\bar{\alpha}_{1}^{i}\bar{\alpha}_{2}^{j}\bar{\beta}_{1}^{k}\bar{\beta}_{2}^{l}\gamma^s)=\alpha_{1}^{i}\alpha_{2}^{j}\beta_{1}^{k}\beta_{2}^{l}\gamma_1^s$. 
	If $u=\bar{\alpha}_{1}^{i}\bar{\alpha}_{2}^{j}\bar{\beta}_{1}^{k}\bar{\beta}_{2}^{l}\gamma^s$ and $v=\bar{\alpha}_{1}^{i'}\bar{\alpha}_{2}^{j'}\bar{\beta}_{1}^{k'}\bar{\beta}_{2}^{l'}\gamma^t$ belongs to $G_1/Z_1$, then it can be readily verified that 
	$$s_{1}(u)s_{1}(v)s_{1}(uv)^{-1} =\gamma_1^{-ki'} \gamma_2^{-lj'} (\gamma_1)^{ki'-lj'}=(\gamma_1\gamma_2)^{-lj'}.$$ 
	Thus 
	$$\operatorname{tra}_{1}(\chi)(u,v)=\chi(\gamma_{1}\gamma_2)^{{-lj'}} \text{ for } \chi \in  \mathrm{Hom}(Z_{1}, \mathbb C^\times).$$
It is also easy to see that for $u=\bar{\alpha}_{1}^{i}\bar{\alpha}_{2}^{j}\bar{\beta}_{1}^{k}\bar{\beta}_{2}^{l}\gamma^s$ and $v=\bar{\alpha}_{1}^{i'}\bar{\alpha}_{2}^{j'}\bar{\beta}_{1}^{k'}\bar{\beta}_{2}^{l'}\gamma^t$ in $G_2/Z_2$,
$$\operatorname{tra}_{2}(\chi)(u,v)=\chi(\gamma'_{1}\gamma'_2)^{{-lj'}} \text{ for } \chi \in  \mathrm{Hom}(Z_{2}, \mathbb C^\times).$$
Hence by defining the isomorphisms $i : Z_{2} \to Z_{1}$ and $\phi :  G_2/Z_{2} \to G_1/Z_{1}$ on the generators by $i(\gamma_1'\gamma_2')=\gamma_1\gamma_2$ and $\phi (\bar{\alpha}_{1}^{i}\bar{\alpha}_{2}^{j}\bar{\beta}_{1}^{k}\bar{\beta}_{2}^{l}\gamma^s)=\bar{\alpha}_{1}^{i}\bar{\alpha}_{2}^{j}\bar{\beta}_{1}^{k}\bar{\beta}_{2}^{l}\gamma^s$
respectively, we see that the induced isomorphisms $\bar{i}$ and $\bar{\phi}$ yield  the  above commutative diagram. 
Consequently, from Corollary \ref{result1}, it immediately follows that $G \sim_\mathbb C H$.

%
	It can be verified that the same arguments holds for the rest of the groups. This completes proof of the theorem.  	
	
\end{proof}

  \subsection{Direct product of groups}\label{section1}
  In this section, we discuss the TGRIP for direct product of groups. \\

\noindent \textbf{Proof of Theorem \ref{direct_product}}
By Theorem \ref{kar}, it follows that 
\begin{eqnarray*}
 \Ho^{2}(G,\mathbb{C}^\times)  \cong \Ho^{2}(G_{1},\mathbb{C}^\times) \times \Ho^{2}(G_{2},\mathbb{C}^{\times}) \times \cdots \times \Ho^{2}(G_{r},\mathbb{C}^\times) \\
\Ho^{2}(H,\mathbb{C}^\times)  \cong \Ho^{2}(H_{1},\mathbb{C}^\times) \times \Ho^{2}(H_{2},\mathbb{C}^\times) \times \cdots \times \Ho^{2}(H_{r},\mathbb{C}^\times)
\end{eqnarray*}

 Suppose $G_{i} \sim_{\mathbb{C}} H_{i}$. Then there is an isomorphism $\psi_{i}:~\Ho^{2}(G_{i},\mathbb{C}^{\times}) \rightarrow \Ho^{2}(H_{i},\mathbb{C}^{\times})$ such that for every $[\alpha] \in \Ho^{2}(G_{i},\mathbb{C}^{\times})$, $\mathbb{C}^{\alpha}G_{i} \cong \mathbb{C}^{\psi_{i}(\alpha)}H_{i}$ for $1 \leq i\leq r$. 
	Consider the group isomorphism 
	$$\psi : =\psi_{1}\times \psi_{2} \times \cdots \times \psi_{n}:~
	\Ho^{2}(G,\mathbb{C}^{\times}) \to \Ho^{2}(H,\mathbb{C}^{\times}).$$
	Let $[\beta] \in \Ho^{2}(G,\mathbb{C}^{\times})$ and $[\beta_{i}]=\mathrm{res}^G_{G_i}([\beta])$.
 As $G_{i} \sim_{\mathbb{C}} H_{i}$, we have that $\mathbb{C}^{\beta_{i}}G_{i} \cong \mathbb{C}^{\psi_{i}(\beta_{i})}H_{i}$ for $1 \leq i \leq r.$ Further, it follows from \cite[Proposition 1.1, page 196]{Karpilovsky} that 
	$$\mathbb{C}^{\beta}G \cong \mathbb{C}^{\beta_{1}}G_{1} \otimes \mathbb{C}^{\beta_{2}}G_{2} \otimes \cdots \otimes \mathbb{C}^{\beta_{r}}G_{r} $$ 
	and 
	$$\mathbb{C}^{\psi(\beta)}H \cong \mathbb{C}^{\psi_{1}(\beta_{1})}H_{1} \otimes \mathbb{C}^{\psi_{2}(\beta_{2})}H_{2} \otimes \cdots \otimes \mathbb{C}^{\psi_{r}(\beta_{r})}H_{r}.$$ Consequently, $\mathbb{C}^{\beta}G \cong \mathbb{C}^{\psi(\beta)}H$.

Conversely, suppose $G \sim_{\mathbb{C}} H$. Then there is an isomorphism $\psi:~ \Ho^{2}(G,\mathbb{C}^{\times}) \to \Ho^{2}(H,\mathbb{C}^{\times})$ such that $\mathbb{C}^{\alpha} G \cong \mathbb{C}^{\psi(\alpha)} H$ for  $[\alpha] \in \Ho^{2}(G,\mathbb{C}^{\times})$. 
Let $\psi_{i}$ denotes the restriction of $\psi$ on $ \Ho^{2}(G_i,\mathbb{C}^{\times})$ for $1 \leq i \leq r$. Observe that $\psi_i:~ \Ho^{2}(G_i,\mathbb{C}^{\times}) \to \Ho^{2}(H_i,\mathbb{C}^{\times})$ are isomorphisms as $(k_{i},k_{j})=1$. Let $[\beta_i] \in \Ho^{2}(G_{i},\mathbb{C}^{\times})$ and we denote $[\beta]=([1],[1],\cdots, [\beta_i], [1], \cdots, [1]) \in \Ho^{2}(G,\mathbb{C}^{\times})$. Then we have 
\begin{eqnarray*}
\mathbb{C}^{\beta}G & \cong & \mathbb{C}G_{1} \otimes \mathbb{C}G_{2} \otimes \cdots \otimes  \mathbb{C}^{\beta}G_{i}  \otimes 
 \cdots \otimes  \mathbb{C}G_{r}, \\
 \mathbb{C}^{\psi(\beta)}H & \cong &  \mathbb{C}H_{1} \otimes \mathbb{C}H_{2} \otimes \cdots \otimes  \mathbb{C}^{\psi_i(\beta)}H_{i}  \otimes 
 \cdots \otimes  \mathbb{C}H_{r}.
\end{eqnarray*}
Since $ \mathbb{C}^{\beta}G \cong \mathbb{C}^{\psi(\beta)}H$ and $(k_{i},k_{j})=1$, by dimension consideration, we have $ \mathbb{C}^{\beta}G_{i}  \cong \mathbb{C}^{\psi_i(\beta)}H_{i}$.
This proves that $G_{i} \sim_{\mathbb{C}} H_{i}$ for $1 \leq i \leq r$.
 Hence the result follows.
 
 \qed

Based on the above result, it follows that it is enough to consider $p$-groups.
We propose the following three conditions that relate to the TGRIP of the group and the TGRIP of their components.

\begin{definition}
We say that $p$-groups $G=G_{1}\times G_{2}$ and $H=H_{1}\times H_{2}$ satisfy conditions (A), (B), or (C) respectively, if 
\begin{itemize} \item [(A)] $G_{i} \sim_\mathbb{C} H_{i}$ for $i=1,2$.
\item [(B)] $G_{i}/G_{i}' \cong H_{i}/H_{i}'$ for $i=1,2$. 
\item [(C)] $G \sim_\mathbb{C}H$. 
\end{itemize}
\end{definition}

Now, we provide several examples to show the relations between (A), (B), and (C).\\

\noindent \textbf{Example 1:}
The following example says that (A) does not imply (B), (C).
Consider the  groups of order $p^4$ 
\[
G_1=\langle a, b, c\mid [b,c]=a^p, a^{p^2}=b^p=c^p=1 \rangle,
\]
\[
H_1=\langle a, b, c\mid [a,c]=b, a^{p^2}=b^p=c^p=1 \rangle.
\]
It follows from \cite[Theorem 4.3]{MS} that $G_1 \sim_\mathbb C H_1$. 
Suppose $G=G_1 \times C_p$ and $H=H_1 \times C_p$. 
Note that $G_{1}/G_{1}' \cong C_{p} \times C_{p} \times C_{p}$, whereas $H_{1}/H_{1}' \cong C_{p^{2}} \times C_{p}$.
Hence, due to Theorem \ref{kar}, it follows that $\Ho^2(G,\mathbb C^\times)$ and $\Ho^2(H,\mathbb C^\times)$ are not isomorphic. 
 Therefore, $G \nsim_\mathbb C H$, and conditions (B), (C) do not hold in this case.\\

 \textbf{Example 2:}
 Consider the following groups of order $p^5$. In this case (C) holds, but (A) and (B) do not hold.
 \[
 G=\langle a, b\mid [b,a]=a^p, a^{p^2}=b^p=1\rangle \times C_{p^2},
 \]
 \[
 H=\langle a, b\mid [b,a]=a^{p^2}, a^{p^3}=b^p=1\rangle \times C_p.
 \]
 Taking $G_1=\langle a, b\mid [b,a]=a^p, a^{p^2}=b^p=1\rangle$, $H_1=\langle a, b\mid [b,a]=a^{p^2}, a^{p^3}=b^p=1\rangle$ and $G_2=C_{p^2}$, $H_2=C_p$,  we have  $G_i \nsim_\mathbb C H_i$ for $i=1,2$. But $G \sim_\mathbb C H$, follows from Theorem \ref{groupsorderp5} (i).
 \\

\textbf{Example 3:}
In the  following example   $(A)$, $(B)$ and $(C)$ holds.
Consider the following groups of order $p^4$: 
\[
G_1=\langle a, b, c\mid [b,c]=a^p, a^{p^2}=b^p=c^p=1 \rangle,
\]
\[
H_1=\langle \bar{a},  \bar{b},  \bar{c}\mid [ \bar{a}, \bar{c}]= \bar{a}^p,  \bar{a}^{p^2}= \bar{b}^p= \bar{c}^p=1 \rangle.
\]
It follows from \cite[Theorem 4.3]{MS} that $G_1 \sim_\mathbb C H_1$. 
It is also easy to check that $G_1, H_1$ satisfy the hypotheses of Corollary \ref{result1} by taking the central subgroups  $G_1', H_1'$ respectively.
Thus, by Corollary \ref{coro1},
$G_1 \times A \sim_\mathbb C H_1 \times A$, for every finite abelian group $A$.

Now, we prove Corollary \ref{coro1} which describes that  $(A)$ implies $(C)$ under certain conditions.\\

\noindent\textbf{Proof of Corollary \ref{coro1}}
 (1) Consider $G=G_1\times G_2$ and $H=H_1\times H_2$. Since $G_i$ and $H_i$ satisfy the hypotheses of Corollary \ref{result1}, there are  central subgroups $Z_i, \bar{Z_i}$ of $G_i$ and $H_i$ respectively, and there are isomorphisms $\phi_i: H_i/\bar{Z}_i\to G_i/Z_i$ and $j_i: \bar{Z}_i \to Z_i$ such that $\bar{\phi_i}\circ \tra_1=\tra_2\circ \bar{j_i}$, for the induced isomorphisms  $\bar{\phi_i}$ and $\bar{j_i}$ for $i=1,2$.
Thus, we have the following commutative diagram.
	\[
\xymatrix{ 
	1 \ar[r] &  \mathrm{Hom}(Z_{i}, \mathbb C^\times) \ar[d]^{\bar{j_i}} \ar[r] ^{\mathrm{tra}_{1}} & \Ho^2(G_i/Z_{i}, \mathbb C^\times)  \ar[d]^{\bar{\phi_i}} \\
	1 \ar[r] &  \mathrm{Hom}(\bar{Z_{i}}, \mathbb C^\times) \ar[r] ^{\mathrm{tra}_{2}} & \Ho^2(H_i/\bar{Z_{i}}, \mathbb C^\times).\\
}
\]
By \cite[Theorem 1.5.1]{GK}, it follows that $\res^{G_i}_{Z_i}: \Ho^2(G_i, \mathbb C^\times) \to \Ho^2(Z_{i}, \mathbb C^\times)$ is trivial which implies that $Z_i\subseteq G_i'$, and similarly $\bar{Z_i}\subseteq H_i'$.
We want to show that $G$ and $H$ satisfy the hypotheses of Corollary \ref{result1}.  
 Taking $\phi=(\phi_1, \phi_2)$ and $j=(j_1,j_2)$
 we have the following commutative diagram for the induced isomorphisms $\bar{\phi}$ and $\bar{j}$ with exact rows.
 	\[
	\xymatrix{ 
		1 \ar[r] &  \mathrm{Hom}(Z_{1}\times Z_2, \mathbb C^\times) \ar[d]^{\bar{j}} \ar[r]^{\mathrm{tra}_{1}\;\;\;\;\;} & \Ho^2(G_1/Z_{1}\times G_2/Z_2, \mathbb C^\times)  \ar[d]^{\bar{\phi}} \ar[r]^{ \;\;\;\;\;\;\;\;\;\;\;\; \mathrm{inf}_{1}} & \Ho^2(G, \mathbb C^\times) \ar[r] &  1 \\
		1 \ar[r] &  \mathrm{Hom}( \bar{Z_1}\times  \bar{Z_2}, \mathbb C^\times) \ar[r] ^{\mathrm{tra}_{2}\;\;\;\;} & \Ho^2(H_1/\bar{Z_1}\times H_2/\bar{Z_2}, \mathbb C^\times)\ar[r]^{\;\;\;\;\;\;\;\;\;\;\;\;\; \mathrm{inf}_{2}} & \Ho^2(H, \mathbb C^\times) \ar[r] &  1 
	}
	\]
	Hence the result follows.\\
	
	(2) The proof immediately follows from the fact that $G_1\times  A$ and $H_1\times  A$ satisfy the hypothesis of Corollary \ref{result1}. 
 
 \qed


%
%
%

\subsection{Central product of groups}\label{section2}
 In this section, we study TGRIP for the central product of groups.\\
 
 \noindent \textbf{Proof of Theorem \ref{central product}}
	If $G$ is a central product of $G_1$ and $G_2$ and  $Z_1=G_{1}'\cap G_{2}'$, then it follows from  \cite[Theorem A]{centralproduct} that the following sequence
	\[
		1 \to  \mathrm{Hom}(Z_{1}, \mathbb C^\times) \xrightarrow{\mathrm{tra}}  \Ho^2(G/Z_{1}, \mathbb C^\times)    \xrightarrow{\mathrm{inf}}  \Ho^2(G, \mathbb C^\times) \to  1. 	
	\]
	is exact.
	Now, by Theorem \ref{mainthm}, the result follows.
	
\qed

As an application of this result, we provide some examples.
Consider the groups of order $p^6$ belonging to the isoclinism class $\Phi_5$ mentioned in \cite{James}. 
\begin{thm}\label{example2}
Let $G,H$ be two groups of order $p^6$ belonging to the set $\{\Phi_5(311), \Phi_5(2211)a, $
$\Phi_5(2211)b, \Phi_5(21^4)c\}$. Then $G \sim_\mathbb C H$.
\end{thm}
\begin{proof}
	Let $$G=\Phi_{5}(21^4)c=\langle \alpha_i, \beta_1; 1 \leq i \leq 4 \mid [\alpha_1,\alpha_2]= [\alpha_3,\alpha_4] =\beta_1, \alpha_1^{p^2}=\alpha_2^p=\alpha_3^p=\alpha_4^p=\beta_1^p=1 \rangle$$ and 
	$$H=\Phi_{5}(2211)b=\langle \alpha^{\prime}_i, \beta_2; 1 \leq i \leq 4 \mid [\alpha^{\prime}_1,\alpha^{\prime}_2]= [\alpha^{\prime}_3,\alpha^{\prime}_4] ={\alpha_3^{\prime}}^p=\beta_2, {\alpha_1^{\prime}}^{p^2}={\alpha_2^{\prime}}^p={\alpha_4^{\prime}}^p=\beta_2^p=1  \rangle.$$ 
	It is easy to see that $G$ is a central product of $G_{1}$ and $G_{2}$ for
	$$G_{1}=\langle \alpha_{1},\alpha_{2}, \beta_1~|~[\alpha_{1},\alpha_{2}]=\beta_{1},~\alpha_{1}^{p^{2}}=\alpha_{2}^{p}=\beta_{1}^{p}=1 \rangle \text{ and }$$
	$$G_{2}=\langle \alpha_3,\alpha_4, \beta_1\mid [\alpha_{3},\alpha_{4}]=\beta_{1},\alpha_{3}^{p}=\alpha_{4}^{p}=\beta_{1}^{p}=1 \rangle.$$ Further, $H$ is a central product of $H_{1}$ and $H_{2}$, where 
	$$H_{1} =\langle \alpha_{1}^{\prime},\alpha_{2}^{\prime}, \beta_2~|~[\alpha^{\prime}_1,\alpha^{\prime}_2]=\beta_{2},~{\alpha_1^{\prime}}^{p^2}={\alpha_2^{\prime}}^p=\beta_{2}^{p}=1\rangle \text{ and }$$
	$$H_{2}=\langle \alpha_{3}^{\prime},\alpha_{4}^{\prime}, \beta_2~|~[\alpha^{\prime}_3,\alpha^{\prime}_4] ={\alpha_3^{\prime}}^p=\beta_2, {\alpha_4^{\prime}}^p=\beta_2^p=1 \rangle.$$ 
	Following the notations provided in Theorem \ref{central product}, we have $Z_{1} \cong Z_{2} \cong C_{p}$ and $G/Z_{1} \cong H/Z_{2} \cong C_{p^{2}} \times (C_{p})^3 $.  
	Consider the following diagram
	\[
	\xymatrix{
		1 \ar[r] &  \mathrm{Hom}(Z_{1}, \mathbb C^\times) \ar[d]^{\bar{i}}\ar[r] ^{\mathrm{tra}_1} & \Ho^2(G/Z_{1}, \mathbb C^\times) \ar[r]\ar[r]^{\mathrm{inf_1}} \ar[d]^{\bar{\phi}}& \Ho^2(G, \mathbb C^\times) \ar[r] &  1,\\
		1 \ar[r] &  \mathrm{Hom}(Z_{2}, \mathbb C^\times) \ar[r] ^{\mathrm{tra}_2} & \Ho^2(H/Z_{2}, \mathbb C^\times) \ar[r]\ar[r]^{\mathrm{inf_2}} & \Ho^2(H, \mathbb C^\times) \ar[r] &  1.
	}
	\]  
	Here, the transgression maps are given as follows:
	define a section $s_{1}: G/Z_{1} \to G$ by $s_{1}(\alpha_{1}^{i}\alpha_{2}^{j}\alpha_{3}^{k}\alpha_{4}^{l}Z_{1})=\alpha_{1}^{i}\alpha_{2}^{j}\alpha_{3}^{k}\alpha_{4}^{l}$. For $u=\alpha_{1}^{i}\alpha_{2}^{j}\alpha_{3}^{k}\alpha_{4}^{l}Z_{1}$ and $v=\alpha_{1}^{i'}\alpha_{2}^{j'}\alpha_{3}^{k'}\alpha_{4}^{l'}Z_{1}$, we have $s_{1}(u)s_{1}(v)s_{1}(uv)^{-1} =\beta_{1}^{-(i'j+k'l)}$. Hence 
	$$\operatorname{tra}_{1}(\chi)(u,v)=\chi(\beta_{1})^{-(i'j+k'l)} \text{ for } \chi \in  \mathrm{Hom}(Z_1, \mathbb C^\times).$$
	Define a section $s_{2}: H/Z_{2} \to H$ by $s_{2}({\alpha_{1}^\prime}^{i}{\alpha_{2}^\prime}^{j}{\alpha_{3}^\prime}^{k}{\alpha_{4}^\prime}^{l}Z_{2})={\alpha_{1}^\prime}^{i}{\alpha_{2}^\prime}^{j}{\alpha_{3}^\prime}^{k}{\alpha_{4}^\prime}^{l}$. Following the same arguments as mentioned above for the group $G$, we obtain that for $u={\alpha_{1}^\prime}^{i}{\alpha_{2}^\prime}^{j}{\alpha_{3}^\prime}^{k}{\alpha_{4}^\prime}^{l}Z_{2}$ and $v={\alpha_{1}^\prime}^{i'}{\alpha_{2}^\prime}^{j'}{\alpha_{3}^\prime}^{k'}{\alpha_{4}^\prime}^{l'}Z_{2}$,  $s_{2}(u)s_{2}(v)s_{2}(uv)^{-1} =\beta_{2}^{-(i'j+k'l)}$. Hence 
	$$\operatorname{tra}_{2}(\chi)(u,v)=\chi(\beta_{2})^{-(i'j+k'l)} \text{ for } \chi \in  \mathrm{Hom}(Z_2, \mathbb C^\times).$$

Define isomorphisms $i : Z_{2} \to Z_{1}$ and $\phi : H/Z_{2} \to G/Z_{1}$ by $i(\beta_{2})=\beta_{1}$ and  $\phi(\alpha^\prime_{i}Z_{2})=\alpha_{i}Z_{1}$ for $1 \leq i \leq 4$ respectively. 
		Then, the induced isomorphisms $\bar{i}$ and $\bar{\phi}$ yield the  commutative diagram. Hence, from Theorem \ref{central product}, it immediately follows that $G \sim_\mathbb C H$. 
		
The same arguments hold for the remaining groups in the list. This finishes the proof.  
	
\end{proof}

\noindent \textbf{Proof of Theorem \ref{extra_special}}
If $G$ and $H$ are extra-special $p$-groups of same order, then $\mathbb{C}G \cong \mathbb{C}H$, due to \cite[Theorem 2.18, page 813]{GK1}.

$(i)$ For $n=1$, by \cite[Theorem 3.3.6]{GK}, it follows that the groups have non-isomorphic Schur multiplier. Hence, the result follows.

$(ii)$ Assume $n>1$. 
Since $G$ is an extra-special $p$-group of order $p^{2n+1}$, by \cite[Theorem 3.3.4]{GK}, $G$ is a central product of $n$ non-abelian groups of order $p^3$. 
Due to \cite[Theorem 3.14]{craven}, there are two non-isomorphic extra-special $p$-groups.
Thus, if $G$ is of exponent $p$, then $G=\prod_{i=1}^n G_i$, where $G_i=\langle x_i, y_i \mid [x_i, y_i]=z, x_i^p=y_i^p=z^p=1\rangle.$
Hence $G'=\langle z \rangle\cong C_p$ and $G/G' \cong \langle x_i G'; 1 \leq i \leq n\rangle \times  \langle y_i G'; 1 \leq i \leq n\rangle\cong (C_p)^{2n}$. 
Further, $[x_i,y_i]=z, 1\leq i \leq n$ and $[x_i, x_j]=[y_i,y_j]=1, 1\leq i, j \leq n$. 
Now
consider the transgression map $\tra_1: \mathrm{Hom}(G', \mathbb C^\times) \to \Ho^2(G/G', \mathbb C^\times)$ which is defined as follows: take a section $s: G/G' \to G$ defined by 
$$s(\prod_{k=1}^n x_k^{i_k}\prod_{k=1}^n y_k^{j_k}G')=\prod_{k=1}^n x_k^{i_k}\prod_{k=1}^n y_k^{j_k}.$$ Then, for
$X=\prod_{k=1}^n x_k^{i_k}\prod_{k=1}^n y_k^{j_k} G', Y=\prod_{k=1}^n x_k^{i'_k}\prod_{k=1}^n y_k^{j'_k} G'$ in $G/G'$, we have
$$s(X)s(Y)s(XY)^{-1}=z^{-\sum_{k=1}^n i_k'j_k}.$$
Hence 
$$\tra_1(\chi)\big(X, Y\big)=\chi(z)^{-\sum_{k=1}^n i_k'j_k} \text{ for } \chi \in  \mathrm{Hom}(G', \mathbb C^\times).$$
Now, if $H$ is an extra-special $p$-group of exponent $p^2$, then $H=\prod_{i=1}^n H_i$  with $H_i=\langle  \tilde{x}_i, \tilde{y}_i\mid [ \tilde{x}_i,\tilde{y}_i]= \tilde{x}_i^p=\tilde{z},  \tilde{y}_i^p=\tilde{z}^p=1\rangle.$
Therefore, $H'=\langle \tilde{z} \rangle$ and $H/H' \cong \langle \tilde{x}_i H'; 1 \leq i \leq n\rangle \times  \langle \tilde{y}_i H'; 1 \leq i \leq n\rangle\cong (C_p)^{2n}$.
In $H$, we have $[\tilde{x}_i,\tilde{y}_i]= \tilde{z}$ for $1\leq i \leq n$ and $[\tilde{x}_i,\tilde{ x}_j]=[\tilde{y}_i,\tilde{y}_j]=1$ for $1\leq i, j \leq n$. 
So, as described above, it is easy to check that the transgression map $\tra_2: \mathrm{Hom}(H', \mathbb C^\times) \to \Ho^2(H/H', \mathbb C^\times)$ is defined by
$$\tra_2(\chi)\big(\prod_{k=1}^n \tilde{x}_k^{i_k}\prod_{k=1}^n \tilde{y}_k^{j_k} H', \prod_{k=1}^n \tilde{x}_k^{i'_k}\prod_{k=1}^n \tilde{y}_k^{j'_k}H'\big)=\chi(\tilde{z})^{-\sum_{k=1}^n i_k'j_k}  \text{ for } \chi \in  \mathrm{Hom}(H', \mathbb C^\times).$$
Consider
the isomorphisms $i: H'\to G'$ and $\phi :~H/H' \to G/G'$ defined on the generators by $i(\tilde{z})=z$  and   $\phi(\tilde{x}_iH')=x_i G', \phi(\tilde{y}_iH')={y}_i G'$ for $1 \leq i \leq n$, respectively. Then the induced isomorphisms $\bar{i}$ and $\bar{\phi}$ yields the following commutative diagram 
 \[
	\xymatrix{ 
		1 \ar[r] &  \mathrm{Hom}(G', \mathbb C^\times) \ar[d]^{\bar{i}} \ar[r] ^{\mathrm{tra}_{1}} & \Ho^2(G/G', \mathbb C^\times)   \ar[d]^{\bar{\phi}}  \\
		1 \ar[r] &  \mathrm{Hom}(H', \mathbb C^\times) \ar[r] ^{\mathrm{tra}_{2}} & \Ho^2(H/H', \mathbb C^\times).
	}
	\]	
Hence the result follows from Theorem \ref{central product}, as $Z_1=G'$ and $Z_2=H'$.

\qed

\section{Groups of order $p^5$, where $p \geq 5$ is prime}\label{groupsp5}
In this section, we study the TGRIP for groups of order $p^5$, where $p \geq 5$ is prime. 
 First, we recall the Clifford theory, which will be used  to describe the irreducible ordinary representations of a group $G$.  Let $N$ be an abelian normal subgroup of $G$ and $\chi \in \irr(N)$. Then the inertia group  $I_G(\chi)=\{g \in  G\mid \chi^g=\chi\}$ is a normal subgroup of $G$ containing $N$, and  it follows by  \cite[Corollary 11.22]{Isaacs-book} that $\chi$ is extendible to a subgroup $N_1$ for $N < N_1 < I_G(\chi)$ if $N_1/N$ is cyclic.
 By \cite[Theorem 6.11]{Isaacs-book}, there is a bijective correspondence between the sets 
$\irr(I_G(\chi)\mid \chi)$ and $\irr(G\mid \chi)$ via the map $\theta \mapsto \mathrm{Ind}^G_{I_G(\chi)}(\theta)$. We also refer to \cite[Theorem 2.3]{PGS} for these results.
We use these facts in the proof of Theorem \ref{groupsorderp5} without further reference.

Recall that a finite group $G$ is of central type if there is a cocycle $\alpha\in Z^2(G,\mathbb C^\times)$ such that $G$ has a unique irreducible $\alpha$-representation, i.e., the twisted group algebra $\mathbb C^\alpha G$ is simple.

	\begin{thm}\label{centraltype}
	Let $G$ be a non-abelian group of order $p^5$ and $Z$ be a central subgroup of $G$ such that the following conditions are satisfied.
	\begin{enumerate}[(i)]
		\item $G/Z$ is not of central type and $|G/Z|\geq p^2$.
		\item The sequence $1 \to \Hom(Z, \mathbb C^\times)  \xrightarrow{\mathrm{tra}} \Ho^2(G/Z, \mathbb C^\times)  \xrightarrow{\mathrm{inf}} \Ho^2(G, \mathbb C^\times) \to 1$ is exact.
	\end{enumerate}
	Then for each non-trivial $[\alpha] \in \Ho^2(G,\mathbb C^\times)$, $\mathbb C^\alpha G \cong \oplus_{i=1}^{p^3} \mathbb C^{p\times p}.$
\end{thm}
\begin{proof}
	 Given  $[\alpha] \in \Ho^2(G, \mathbb C^\times)$, we have 
	\[
\mathbb C^\alpha G \cong \prod_{\{\beta \in \Ho^2(G/Z, \mathbb C^\times)\mid \inf([\beta])=\alpha\}} \mathbb C^\beta [G/Z],	
	\]
due to Theorem \ref{dimpreserve}.
Consequently, the result follows as $p^2 \leq |G/Z| \leq p^4$ and $G/Z$ is not of central type.

	\end{proof}
Before stating our next result, we recall the Schur multiplier  $\Ho^2(G,\mathbb C^\times)$ of non-abelian $p$-groups $G$ of order $p^5$ , where $p\geq 5$ is an odd prime. 
\begin{thm}[Table 2 of \cite {groupsp5}]\label{Schurp5}
Let $G$ be a non-abelian group of order $p^5, p\geq 5$, for odd prime $p$. 
	\begin{enumerate}[(a)]
\item $ \Ho^2(G,\mathbb C^\times)=1$ if and only if $G$ is isomorphic to one of the following groups: 
$\Phi_2(41), $ $\Phi_6(221)a, $ $\Phi_6(221)b_r, r\neq \frac{p-1}{2},  \Phi_6(221)c_r, \Phi_6(221)d_r, \Phi_8(32)$.\\

\item $ \Ho^2(G,\mathbb C^\times) \cong \mathbb Z/p\mathbb Z$ if and only if $G$ is isomorphic to one of the following groups:  $\Phi_2(32)a_1, \Phi_2(32)a_2,$
$\Phi_3(311)a,\Phi_3(311)b_r, \Phi_3(221)a,  \Phi_4(221)a,\Phi_4(221)c,$ $ \Phi_4(221)d_r, r\neq \frac{p-1}{2}, \Phi_4(221)e, \Phi_4(221)f_r, \Phi_6(221)_{b_{\frac{p-1}{2}}}, \Phi_6(221)d_0, \Phi_6(2111)a,  \Phi_6(2111)b_r,\Phi_9(2111)a, \\
\Phi_9(2111)b_r, \Phi_{10}(2111)a_r, \Phi_{10}(2111)b_r$.\\

\item $ \Ho^2(G,\mathbb C^\times) \cong \mathbb Z/p\mathbb Z \times \mathbb Z/p\mathbb Z$ if and only if $G$ is isomorphic to one of the following groups: 
$\Phi_2(311)a, \Phi_2(221)b, \Phi_2(311)b, \Phi_2(311)c, \Phi_3(221)b_r, \Phi_3(2111)d, \Phi_3(2111)e, \Phi_4(221)b.$
\\

\item $ \Ho^2(G,\mathbb C^\times) \cong (\mathbb Z/p\mathbb Z)^3$ if and only if $G$ is isomorphic to one of the following groups: \\$\Phi_2(221)a, \Phi_2(221)d, \Phi_3(2111)a, \Phi_3(2111)b_r, \Phi_3(2111)c, \Phi_4(2111)a, \Phi_4(2111)b, \Phi_4(2111)c,$ $ \Phi_6(1^5), \Phi_7(2111)a, \Phi_7(2111)b_r, \Phi_7(2111)c, \Phi_9(1^5), \Phi_{10}(1^5)$.\\

\item $ \Ho^2(G,\mathbb C^\times)\cong (\mathbb Z/p\mathbb Z)^4$ if and only if $G$ is isomorphic to one of the following groups: \\
$\Phi_2(2111)c, \Phi_2(2111)d, \Phi_{3}(1^5), \Phi_{7}(1^5)$.\\

\item $ \Ho^2(G,\mathbb C^\times) \cong (\mathbb Z/p\mathbb Z)^5$ if and only if $G$ is isomorphic to one of the following groups: \\
$\Phi_2(2111)a, \Phi_2(2111)b,  \Phi_5(2111), \Phi_{5}(1^5)$.\\

\item $ \Ho^2(G,\mathbb C^\times) \cong (\mathbb Z/p\mathbb Z)^6$ if and only if $G$ is isomorphic to $\Phi_{4}(1^5)$.\\

\item $ \Ho^2(G,\mathbb C^\times) \cong (\mathbb Z/p\mathbb Z)^7$ if and only if $G$ is isomorphic to $\Phi_{2}(1^5)$.\\

\item $ \Ho^2(G,\mathbb C^\times) \cong \mathbb Z/p^2\mathbb Z$ if and only if $G$ is isomorphic to $\Phi_{4}(221)_{d_{\frac{p-1}{2}}}, \Phi_{4}(221)_{f_0}$.\\

\item $ \Ho^2(G,\mathbb C^\times) \cong \mathbb Z/p^2\mathbb Z \times  \mathbb Z/p\mathbb Z$ if and only if $G$ is isomorphic to $\Phi_2(221)c$.\\
	\end{enumerate}
\end{thm}

Let $S$ denote the set $\{\Phi_3(221)b_r,\Phi_4(221)b, \Phi_4(221)d_{\frac{1}{2}(p-1)}, 
\Phi_4(221)f_0,   \Phi_{6}(1^5)\}$ consisting of some groups of order $p^5$. It follows by \cite[Lemma 1.2]{MS} that finite abelian groups form singleton equivalence classes w.r.t the relation $\sim_\mathbb C$. Hence, in the next result, we answer the TGRIP for the non-abelian groups of order $p^5$ which do not belong to the set $S$.

 \begin{thm}\label{groupsorderp5}
The non-abelian groups of order $p^5$, where $p \geq 5$ is prime, with the exceptions listed in $S$ consist of the following equivalence classes w.r.t the relation $\sim_\mathbb C$.
	\begin{enumerate}[(i)]
			\item $\{\Phi_2(221)b,  \Phi_2(311)a, \Phi_2(311)b, \Phi_2(311)c\}$, 
					\item  $\{\Phi_2(32)a_1, \Phi_2(32)a_2\}$, 
\item  $\{\Phi_2(41)\}$,
\item  $ \{\Phi_2(2111)a,\Phi_2(2111)b\}$,
\item  $ \{\Phi_2(2111)c,\Phi_2(2111)d\}$, 
\item  $\{\Phi_2(221)a, \Phi_2(221)d, \}$, 
\item $\{\Phi_2(221)c\}$,
		\item  $ \{	\Phi_2(1^5)\}$,
		\item $\{ \Phi_3(311)a, \Phi_3(311)b_r,  \Phi_3(221)a,  \Phi_4(221)a, \Phi_4(221)c, \Phi_4(221)d_r  (r \neq \frac{1}{2}(p-1)),  \Phi_4(221)e, \\
  \Phi_4(221)f_r\},$
		\item $\{ \Phi_3(2111)a, \Phi_3(2111)b_r, \Phi_3(2111)c,  \Phi_4(2111)a,  \Phi_4(2111)b,  \Phi_4(2111)c\}$,
\item $\{\Phi_3(2111)d, \Phi_3(2111)e\}$,
		\item $\{\Phi_3(1^5)\}$,
				\item $\{\Phi_4(1^5)\}$,
		\item $\{\Phi_5(2111),\Phi_5(1^5)\}$,
		\item  $\{ \Phi_6(221)a, \Phi_6(221)b_r  (r \neq \frac{1}{2}(p-1)),  \Phi_6(221)c_r, \Phi_6(221)d_r\}$,
		\item $\{  \Phi_6(2111)a,  \Phi_6(2111)b_r,  \Phi_9(2111)a, \Phi_9(2111)b_r\}$,  
  \item $\{\Phi_6(221)b_{\frac{1}{2}(p-1)},  \Phi_6(221)d_0\}$,
\item   $\{ \Phi_7(2111)a, \Phi_7(2111)b_r,  \Phi_{7}(2111)c\}$,
\item $\{\Phi_7 (1^5)\}$, 
\item  $\{ \Phi_8(32)\}$,
\item  $\{\Phi_{9}(1^5)\}$, 
\item $\{ \Phi_{10}(2111)a_r, \Phi_{10}(2111)b_r\}$,
\item  $\{\Phi_{10}(1^5)\}$.
\end{enumerate}
\end{thm}
\noindent \begin{proof}
First, we separate the isoclinism classes of the groups into several equivalence classes  w.r.t $\sim_\mathbb C$ by observing the following facts:\\
(a) Isomorphism of complex group algebras;\\
(b) Schur multiplier.\\
Note that the Schur multiplier of the groups of order $p^5$, where $p \geq 5$ is prime, has been described in Theorem \ref{Schurp5}. 
One can check that the groups in each list given in $(i)-(xxiii)$ have isomorphic Schur multiplier.
Given  two non-abelian groups $G,H$ of order $p^5$, where $p \geq 5$ is a prime, it follows from  \cite[Theorem 2.2]{Bioch}, \cite[Table 4.1]{James} and   \cite{PDG} that, the complex group algebra $\mathbb{C}G$ is isomorphic to $\mathbb{C}H$ if and only if one of the following holds:\\
(a) $G, H$ are in the same isoclinism class.\\
(b)  $G \in \Phi_{3}$ and $H \in \Phi_{4}$.\\
(c)  $G \in \Phi_6$ and $ H\in \Phi_9$.\\
(d) $G \in \Phi_{7}$ and $H \in \Phi_{8}$.\\
We  use the above facts to obtain the equivalence classes w.r.t $\sim_\mathbb C$.
For example,  one can see that the groups $\Phi_2(41), \Phi_2(221)c,	\Phi_2(1^5), $
$\Phi_3(1^5), \Phi_4(1^5), \Phi_7 (1^5),$ $ \Phi_8(32), \Phi_{9}(1^5),  \Phi_{10}(1^5)$ form singleton equivalence classes w.r.t $\sim_\mathbb C$.

Further, in our proof, we separate some of the equivalence classes 
 by observing the degrees of the ordinary irreducible representations of representation group of $G$. For example, $(xvi)$ and $(xvii)$ have been separated using this fact. Remaining equivalence classes have been separated by using Corollary \ref{result1} and some other facts. \\

\noindent $(i)$ By Theorem \ref{Schurp5}, $\Ho^2(G, \mathbb C^\times) \cong C_p \times  C_p$ for all the groups $G$ listed in $(i)$.  We claim that in this case, 
$\mathbb C^\alpha G \cong \oplus_{i=1}^{p^3} \mathbb C^{p\times p}$
for every non-trivial cocycle $\alpha$ of $G$.

Let $G=\Phi_2(221)b$. Then by \cite[Table 3]{groupsp5} and Theorem \ref{capable}, we have the following exact sequence
\[
1 \to \Hom(G', \mathbb C^\times)  \xrightarrow{\mathrm{tra}} \Ho^2(G/G', \mathbb C^\times)  \xrightarrow{\mathrm{inf}} \Ho^2(G, \mathbb C^\times) \to 1.
\]
Since $G/G'$ is not of central type, by Theorem \ref{centraltype}, $\mathbb C^\alpha G \cong \oplus_{i=1}^{p^3} \mathbb C^{p\times p}$
for every non-trivial cocycle $\alpha$ of $G$. 

Suppose $G=\Phi_2(311)a$. Consider the following group $$
\tilde{G}=\langle \alpha, \alpha_1, \beta,\beta_i \mid [\alpha_1,\alpha]=\alpha^{p^2}, [\beta,\alpha]=\beta_1, [\beta,\alpha_1]=\beta_2, \alpha^{p^3}=\alpha_1^{p}=\beta^p=\beta_i^p=1, i=1,2 \rangle.
$$
Observe that $\tilde{G}=\langle \beta,  \beta_1,\beta_2  \rangle \rtimes \langle \alpha_1, \alpha \mid  [\alpha_1,\alpha]=\alpha^{p^2}, \alpha^{p^3}=\alpha_1^{p}=1\rangle $ is of order $p^7$.  It is easy to see that $\tilde{G}$
is a representation group of $G$ and $\tilde{G}/\langle \beta_1,\beta_2\rangle \cong G$. Consider the abelian normal subgroup $N=\langle \beta,\beta_1,\beta_2,\alpha^p\rangle$ of $\tilde{G}$ of order $p^5$.
Let $\chi' \in \mathrm{Irr}(\langle \beta_1,\beta_2 \rangle)$ such that $\chi'(\beta_1)=\xi^i, \chi'(\beta_2)=\xi^k$ where $\xi$ is a $p$-th root of unity and $0\leq i,k \leq (p-1)$.  Let $\chi \in \mathrm{Irr}(N)$ such that $\chi |_{\langle \beta_1,\beta_2 \rangle }={\chi'}$.
By Theorem \ref{proj-ord}, it is enough to understand $\irr(\tilde{G}\mid \chi')$. So we use  Clifford theory here.
Now  $g\in \tilde{G}$ can be written as $g=\alpha^m\alpha_1^nn'$, for some $n' \in N$, $0\leq m,n\leq p-1$, and  every $n_1 \in N$ can be written as $n_1=\beta^xh$, for some $h \in Z(\tilde{G})$. Therefore,
\begin{flalign}
	\chi^{\alpha^m\alpha_1^n}(\beta^xh)=\chi(\beta^xh)\nonumber
& \iff \chi(\alpha_1^{-n}\alpha^{-m}\beta^xh\alpha^m\alpha_1^n)=\chi(\beta^xh)\nonumber\\
& \iff \chi(\alpha_1^{-n}\alpha^{-m}\beta^x\alpha^m\alpha_1^n)=\chi(\beta^x).\nonumber
\end{flalign} 
Since
\begin{flalign}
	\alpha_1^{-n}\alpha^{-m}\beta^x\alpha^m\alpha_1^n &=\alpha_1^{-n}(\alpha^{-m}\beta^x)\alpha^m\alpha_1^n\nonumber\\
	& =\alpha_1^{-n}(\beta^x\alpha^{-m}\beta_1^{mx})\alpha^m\alpha_1^n\nonumber\\
	& =\alpha_1^{-n}(\beta^x\alpha_1^n)\beta_1^{mx}\nonumber\\
	& =\alpha_1^{-n}(\alpha_1^n\beta^x\beta_2^{nx})\beta_1^{mx}\nonumber\\
	& =\beta^x\beta_1^{mx}\beta_2^{nx},\nonumber
\end{flalign}
we obtain that $\alpha^m\alpha_1^nn' \in I_{\tilde{G}}(\chi)$ if and only if $\chi(\beta_1^{mx}\beta_2^{nx})=1$, i.e.,  $\xi^{(im+kn)x}=1$. 

If $\chi'$ is trivial, then  $\irr(\tilde{G} \mid \chi')$ are $1$ or $p$-dimensional.
Suppose $\chi'$ is non-trivial. Then the following cases occur:\\
\textbf{Case (a) } $i\neq 0, k \neq 0$, then $|I_{\tilde{G}}(\chi)|=|\langle N, \alpha\alpha_1^{-ik^{-1}} \rangle|=p^6$; \\
\textbf{Case (b)} $i= 0, k \neq 0$, then $|I_{\tilde{G}}(\chi)|=|\langle N, \alpha \rangle|=p^6$; \\
\textbf{Case (c)} $i\neq 0, k= 0$, then $|I_{\tilde{G}}(\chi)|=|\langle N, \alpha_1 \rangle|=p^6$.\\
Thus $\irr(\tilde{G} \mid \chi')$ are $p$-dimensional if $\chi'$ is non-trivial.
Hence  $$\mathbb C \tilde{G} \cong  \oplus_{i=1}^{p^4} \mathbb C \oplus_{j=1}^{(p^5-p^2)} \mathbb C^{p\times p}$$
and by Theorem  \ref{proj-ord},  $\mathbb C^\alpha G \cong \oplus_{i=1}^{p^3} \mathbb C^{p\times p}$
for each non-trivial cocycle $\alpha$ of $G$. 
\\

Now consider the groups $H_1=\Phi_2(311)b$ and $H_2=\Phi_2(311)c$. Then 
\[
\tilde{H_1}=\langle \alpha,\alpha_1,\gamma,\beta_1,\beta_2 \mid [\alpha_1,\alpha]=\gamma^{p^2},[\gamma,\alpha_1]=\beta_1,[\gamma,\alpha]=\beta_2,\gamma^{p^3}=\alpha^p=\alpha_1^p=\beta_1^p=\beta_2^p=1\rangle\]
and 
\[
\tilde{H_2}=\langle \alpha,\alpha_1,\alpha_2, \beta_1,\beta_2 \mid [\alpha_1,\alpha]=\alpha_2,[\alpha_2,\alpha]=\beta_1,[\alpha_2,\alpha_1]=\beta_2,\alpha^{p^3}=\alpha_1^p=\alpha_2^p=\beta_1^P=\beta_2^p=1\rangle
\]
are representation groups of $H_1$ and $H_2$ respectively. Proceeding on the similar lines, as described above, one can see that
$$\mathbb C \tilde{H_i} \cong  \oplus_{i=1}^{p^4} \mathbb C \oplus_{j=1}^{(p^5-p^2)} \mathbb C^{p\times p}, i=1,2.$$
Hence the result follows.\\

$(ii)$ 
Consider the group $G_1=\Phi_2(32)a_1=\langle\alpha,\alpha_1,\alpha_2 \mid [\alpha_1,\alpha]=\alpha^{p^2}=\alpha_2,\alpha_1^{p^2}=\alpha_2^p=1\rangle$.\\
Then 
$$\tilde{G}_1=\Phi_{14}(42)=\langle\alpha_1,\alpha_2,\beta\mid [\alpha_1,\alpha_2]=\beta,\alpha_1^{p^2}=\beta,\alpha_2^{p^2}=\beta^{p^2}=1\rangle$$
is a representation group of $G_1$ of order $p^6$ and $\tilde{G}_1/\langle \beta^p \rangle \cong G_1$. Now consider the abelian normal subgroup $N=\langle \alpha_1^p,\alpha_2^p \rangle$ of $\tilde{G}_1$ of order $p^4$. Take $A_1=\langle \alpha_1^{p^3} \rangle$.
Let  $\chi'\in \mathrm{Irr}(\langle \alpha_1^{p^3} \rangle)$ such that $\chi'(\alpha_1^{p^3})=\xi^i$, where $\xi$ is a $p^{th}$ root of unity, and $\chi\in \mathrm{Irr}(N)$ such that $\chi |_{\langle \alpha_1^{p^3} \rangle}={\chi'}$. By Theorem \ref{proj-ord}, it is enough to understand $\irr(\tilde{G}_1\mid \chi')$. So, we use Clifford theory here.
Now $g \in \tilde{G}_1$ can be written as $g=\alpha_1^m\alpha_2^nn'$, for some $n' \in N$, and any $n_1 \in N$ can be written as $n_1=\alpha_1^{pk}\alpha_2^{pl}$ for $0 \leq k \leq p^3-1$, $0 \leq l \leq p-1$. Now
\begin{flalign}
	\chi^{\alpha_1^m\alpha_2^n}(\alpha_1^{pk}\alpha_2^{pl})=\chi(\alpha_1^{pk}\alpha_2^{pl}) 
	& \iff \chi(\alpha_2^{-n}\alpha_1^{-m}\alpha_1^{pk}\alpha_2^{pl}\alpha_1^m\alpha_2^n)=\chi(\alpha_1^{pk}\alpha_2^{pl}) \nonumber.
\end{flalign}
Since
\begin{flalign}
	\alpha_2^{-n}\alpha_1^{-m}\alpha_1^{pk}\alpha_2^{pl}\alpha_1^m\alpha_2^n
	& =\alpha_2^{-n}\alpha_1^{-m}\alpha_1^{pk}(\alpha_2^{pl}\alpha_1^m)\alpha_2^n\nonumber\\
	& =\alpha_2^{-n}\alpha_1^{-m}\alpha_1^{pk}(\alpha_1^m\alpha_2^{pl}\alpha_1^{-p^3ml})\alpha_2^n\nonumber\\
	& =(\alpha_2^{-n}\alpha_1^{pk})\alpha_2^{pl}\alpha_1^{-p^3ml}\alpha_2^n\nonumber\\
	& =(\alpha_1^{pk}\alpha_2^{-n}\alpha_1^{p^3kn})\alpha_2^{pl}\alpha_1^{-p^3ml}\alpha_2^n\nonumber\\
	& =\alpha_1^{pk}\alpha_2^{pl}\alpha_1^{p^3(kn-ml)},\nonumber
\end{flalign}
we obtain that $\alpha_1^m\alpha_2^nn' \in I_{\tilde{G}_1}(\chi)$ if and only if $\chi(\alpha_1^{p^3(kn-ml)})=1$, i.e., $\xi^{i(kn-ml)}=1$. Now the following cases occur:

\textbf{Case (a)} If $i=0$, then $I_{\tilde{G_1}}(\chi)=\tilde{G_1}$. Number of such $\chi$ is $p^3$. 
There are $p^4$ many one dimensional and $(p^3-p^2)$ many $p$-dimensional irreducible representations.

\textbf{Case (b) :} If $i\neq 0$, then $I_{\tilde{G_1}}(\chi)= N$ is of order $p^6$. There are $(p^4-p^3)$ many such $\chi$. In this case,
there are $(p^2-p)$ many $p^2$-dimensional irreducible representations.

Thus, we have 
$$\mathbb C \tilde{G}_1 \cong \oplus_{i=1}^{p^4} \mathbb C \oplus_{i=1}^{p^3-p^2} \mathbb C^{p\times p}\oplus_{i=1}^{p^2-p} \mathbb C^{p^2\times p^2}.$$
\\

Now consider the group $G_2=\Phi_2(32)a_2=\langle\alpha, \alpha_1,\alpha_2 \mid [\alpha_1,\alpha]=\alpha_1^p=\alpha_2,\alpha^{p^3}=\alpha_2^p=1\rangle$.\\
Let $\tilde{G}_2=\Phi_8(33)=\langle \alpha_1,\alpha_2,\beta \mid [\alpha_1,\alpha_2]=\beta=\alpha_1^p,[\beta,\alpha_2]=\beta^p,\alpha_2^{p^3}=\beta^{p^2}=1\rangle$. So $\tilde{G}_2$ is a group of order $p^6$. Then $\tilde{G}_2$ is a representation group of $G$ and $\tilde{G}_2/\langle \beta^p \rangle \cong G_2$. Now consider the abelian normal subgroup $N=\langle \beta=\alpha_1^p,\alpha_2^p \rangle$ of $\tilde{G}_2$ of order $p^4$. Take $A_2=\langle \alpha_1^{p^2} \rangle$.
Let $\chi'\in \mathrm{Irr}(\langle \alpha_1^{p^2} \rangle)$ such that $\chi'(\alpha_1^{p^2})=\xi^i$, where $\xi$ is a $p$-th root of unity.
Let $\chi \in \mathrm{Irr}(N)$ such that $\chi |_{\langle \beta^p \rangle }={\chi'}$.
Now 
\begin{flalign}
	\chi^{\alpha_1^m\alpha_2^n}(\alpha_1^{pk}\alpha_2^{pl})=\chi(\alpha_1^{pk}\alpha_2^{pl}) 
	& \iff \chi(\alpha_2^{-n}\alpha_1^{-m}\alpha_1^{pk}\alpha_2^{pl}\alpha_1^m\alpha_2^n)=\chi(\alpha_1^{pk}\alpha_2^{pl}) \nonumber.
\end{flalign}
Since
\begin{flalign}
	\alpha_2^{-n}\alpha_1^{-m}\alpha_1^{pk}\alpha_2^{pl}\alpha_1^m\alpha_2^n
    & =\alpha_1^{pk}\alpha_2^{pl}\alpha_1^{p^2(kn-ml)},\nonumber
\end{flalign}
we obtain that, $\alpha_1^m\alpha_2^nn' \in I_{\tilde{G}_2}(\chi)$ if and only if $\chi(\alpha_1^{p^2(kn-ml)})=1$, i.e.,  $\xi^{i(kn-ml)}=1$.
Similarly, as above one can see that
$$\mathbb C \tilde{G}_2 \cong \oplus_{i=1}^{p^4} \mathbb C \oplus_{i=1}^{p^3-p^2} \mathbb C^{p\times p}\oplus_{i=1}^{p^2-p} \mathbb C^{p^2\times p^2},$$
and  by Theorem  \ref{proj-ord},  for each non-trivial cocycle $\alpha$ of $G_i$,
$$\mathbb C^{\alpha} G_i \cong  \oplus_{i=1}^{p} \mathbb C^{p^2\times p^2}.$$
Define  an isomorphism $\sigma: \Hom(A_1, \mathbb C^\times) \to \Hom(A_2, \mathbb C^\times)$ on the generators by
$$\sigma(\chi) (\alpha_1^{p^2})=\chi(\alpha_1^{p^3}).$$  Then, there is a dimension preserving bijective correspondence between the sets 
$\irr(\tilde{G}_1\mid \chi)$ and $\irr(\tilde{G}_2\mid \sigma(\chi))$. Hence, $G_1 \sim_{\mathbb C} G_2$, follows from Lemma \ref{bijection_ord}.
\\

$(v)$ Result follows from  \cite[Theorem 1.3]{PGS}.
\\

$(vi)$ Consider the group $$G_1=\Phi_2(221)a=\langle\alpha,\alpha_1\mid [\alpha_1,\alpha]=\alpha^p,\alpha_1^{p^2}=\alpha^{p^2}=1\rangle\times C_p.$$
Consider 
\begin{eqnarray*}
\tilde{G}_1 &=& \langle\alpha_1,\alpha_2,\beta',\beta_1,\beta_2\mid [\alpha_1,\alpha_2]=\alpha_1^p=\beta,[\beta,\alpha_2]=\beta^p,[\beta',\alpha_1]=\beta_1,[\beta',\alpha_2]=\beta_2,\\
&& \beta^{p^2}=\alpha_2^{p^2}=1,\beta'^p=\beta_1^p=\beta_2^p=1\rangle \\
&=& \langle \beta',\beta_1,\beta_2 \rangle \rtimes \langle\alpha_1,\alpha_2 \mid [\alpha_1,\alpha_2]=\alpha_1^p=\beta,[\beta,\alpha_2]=\beta^p,\beta^{p^2}=\alpha_2^{p^2}=1 \rangle\\
&=& (C_p)^3\rtimes \Phi_8(32).
\end{eqnarray*}
Then $\tilde{G}_1$ is a representation group of $G_1$ of order $p^8$ and $\tilde{G}_1/A_1 \cong G_1$ for $A_1=\langle \beta^p, \beta_1,\beta_2 \rangle$. 
Consider the abelian normal subgroup $N=\langle \beta',\beta,\alpha_2^p,\beta_1,\beta_2\rangle$ of $\tilde{G}_1$ of order $p^6$.
Let $\chi' \in \mathrm{Irr}(A_1)$ such that $\chi'(\beta^p)=\xi^{i_1}, \chi'(\beta_1)=\xi^{i_2}, \chi'(\beta_2)=\xi^{i_3}$ where $\xi$ is a  $p^{th}$ root of unity and $0\leq i_1,i_2,i_3 \leq (p-1)$.  Let $\chi \in \mathrm{Irr}(N)$ such that $\chi |_{\langle\beta^p,\beta_1,\beta_2 \rangle }={\chi'}$.
By Theorem \ref{proj-ord}, it is enough to understand $\irr(\tilde{G}_1\mid \chi')$. So we use Clifford theory here.
Now $g \in \tilde{G}_1$ can be written as $g=\alpha_1^m\alpha_2^nn'$ for some $n' \in N$ and any element of $N$ is of the form $n_1=\beta'^l\beta^k\alpha_2^{ps}h$  for some $h \in Z(\tilde{G}_1)$. Therefore,
\begin{flalign}
	\chi^{\alpha_1^m\alpha_2^n}(\beta'^l\beta^k\alpha_2^{ps}h)=\chi(\beta'^l\beta^k\alpha_2^{ps}h) 
	& \iff \chi(\alpha_2^{-n}\alpha_1^{-m}\beta'^l\beta^k\alpha_2^{ps}\alpha_1^m\alpha_2^n)=\chi(\beta'^l\beta^k\alpha_2^{ps}). \nonumber
\end{flalign}
Since
\begin{flalign}
\alpha_2^{-n}\alpha_1^{-m}\beta'^l\beta^k\alpha_2^{ps}\alpha_1^m\alpha_2^n
 & =\alpha_2^{-n}(\alpha_1^{-m}\beta'^l)\beta^k\alpha_2^{ps}\alpha_1^m\alpha_2^n\nonumber\\
 & =\alpha_2^{-n}(\beta'^l\alpha_1^{-m}\beta_1^{lm})\beta^k\alpha_2^{ps}\alpha_1^m\alpha_2^n\nonumber\\
 & =\alpha_2^{-n}\beta'^l\beta^k(\alpha_1^{-m}\alpha_2^{ps})\alpha_1^m\alpha_2^n\beta_1^{lm}\nonumber\\
 & =(\alpha_2^{-n}\beta'^l)\beta^k \beta^{-mps}\alpha_2^{ps}\alpha_2^n\beta_1^{lm}\nonumber\\
 & =\beta'^l(\alpha_2^{-n}\beta^k)\beta^{-mps}\alpha_2^{ps}\alpha_2^n\beta_1^{lm}\beta_2^{nl}\nonumber\\
 & =\beta'^l(\beta^{knp}\beta^k\alpha_2^{-n})\beta^{-mps}\alpha_2^{ps}\alpha_2^n\beta_1^{lm}\beta_2^{nl}\nonumber\\
 & =\beta'^l\beta^{k}\alpha_2^{ps}\beta^{p(kn-ms)}\beta_1^{lm}\beta_2^{nl},\nonumber
\end{flalign}
we obtain that, $\alpha_1^m\alpha_2^nn' \in I_{\tilde{G}_1}(\chi)$ if and only if $\chi(\beta^{p(kn-ms)}\beta_1^{lm}\beta_2^{nl})=1$, i.e.,  \\
$\xi^{i_1(kn-ms)+i_2lm+i_3nl}=1.$
Now the following cases occur.

\textbf{Case (a)} If $i_j=0$ for $j=1, 2, 3$, then $I_{\tilde{G}_1}(\chi)=\tilde{G}_1$. There are $p^3$
such characters $\chi$ of $N$. 
In this case, there are $p^4$ many one dimensional and $(p^3-p^2)$ many $p$-dimensional irreducible  representations  of $\tilde{G}_1$.

\textbf{Case (b)} Suppose  $i_1$ is non-zero. Taking $k=l=0,s=1$, we have $\xi^{-i_1m}=1$ implies $m=0$. Taking $s=l=0,k=1$, we have $\xi^{i_1n}=1$ implies $n=0$.  So $I_{\tilde{G}_1}(\chi)=N$.
There are $p^4(p^2-p)$ such characters $\chi$ of $N$. 
There are $(p^4-p^3)$ many $p^2$-dimensional irreducible  representations.

\textbf{Case (c) :} Suppose $i_1 = 0$ and  $i_3\neq 0$. Then $I_{\tilde{G}_1}(\chi)=\langle N,  \alpha_1\alpha_2^{--{i_2}{i_3}^{-1}}\rangle$ is of order $p^7$. There are $p^4(p-1)$ such characters $\chi$ of $N$. 
There are $(p^5-p^4)$ many $p$-dimensional irreducible  representations.

\textbf{Case (d) :} Suppose $i_1 = 0$ and  $i_2 \neq 0, i_3=0$. Then $I_{\tilde{G}_1}(\chi)=\langle N, \alpha_2\rangle$ is of order $p^7$. There are $(p^4-p^3)$ such characters $\chi$ of $N$. 
There are $(p^4-p^3)$ many $p$-dimensional irreducible  representations.
\\

Therefore, we have 
$$\mathbb C \tilde{G}_1 \cong  \oplus_{i=1}^{p^4} \mathbb C \oplus_{i=1}^{(p^5-p^2)} \mathbb C^{p\times p} \oplus_{i=1}^{(p^4-p^3)} \mathbb C^{p^2\times p^2}.$$

Now consider the group $$G_2=\Phi_2(221)d=\langle \alpha,\alpha_1,\alpha_2 \mid [\alpha_1,\alpha]=\alpha_2,\alpha^{p^2}=\alpha_1^{p^2}=\alpha_2^p=1 \rangle.
$$ Let 
$\tilde{G}_2=\langle\alpha,\alpha_1,\alpha_2,\alpha_3,\alpha_4 \mid [\alpha_1,\alpha]=\alpha_2,[\alpha_2,\alpha]=\alpha_3,[\alpha_2,\alpha_1]=\alpha_4,\alpha^{p^2}=\alpha_1^{p^2}=\alpha_2^{p^2}=\alpha_3^p=\alpha_4^p=1 \rangle.$
Then $\tilde{G}_2$ is a representation group of $G_2$ of order $p^8$ and $\tilde{G}_2/A_2 \cong G$ for $A_2=\langle \alpha_2^p, \alpha_3, \alpha_4 \rangle$. 
Consider the abelian normal subgroup $N=\langle\alpha^p,\alpha_1^p,\alpha_2,\alpha_3,\alpha_4\rangle$ of $\tilde{G}_1$ of order $p^6$.
Let $\chi' \in \mathrm{Irr}(A_2)$ such that $\chi'(\alpha_2^p)=\xi^{i_1}, \chi'(\alpha_3)=\xi^{i_2}, \chi'(\alpha_4)=\xi^{i_3}$ where $\xi$ is a $p^{th}$ root of unity and $0\leq i_1,i_2,i_3\leq (p-1)$.  Let $\chi \in \mathrm{Irr}(N)$ such that $\chi |_{\langle \alpha_2^p,\alpha_3,\alpha_4 \rangle }={\chi'}$.
Now 
\begin{flalign}
\chi^{\alpha^m\alpha_1^n}(\alpha^{pl}\alpha_1^{pk}\alpha_2^sh)=\chi(\alpha^{pl}\alpha_1^{pk}\alpha_2^sh)
& \iff \chi(\alpha_1^{-n}\alpha^{-m}\alpha^{pl}\alpha_1^{pk}\alpha_2^s\alpha^m\alpha_1^n)=\chi(\alpha^{pl}\alpha_1^{pk}\alpha_2^s).\nonumber
\end{flalign}
Now, it is easy to see that
\begin{flalign}
\alpha_1^{-n}\alpha^{-m}\alpha^{pl}\alpha_1^{pk}\alpha_2^s\alpha^m\alpha_1^n
&=\alpha^{pl}\alpha_1^{pk}\alpha_2^s\alpha_2^{p(mk-ln)}\alpha_3^{ms}\alpha_4^{ns},\nonumber
\end{flalign}
and so we obtain that, $\alpha^m\alpha_1^nn' \in I_{\tilde{G}_2}(\chi)$ if and only if $\chi(\alpha_2^{p(mk-ln)}\alpha_3^{ms}\alpha_4^{ns})=1$, i.e., \\
$\xi^{i_1(mk-ln)+i_2ms+i_3ns}=1$.

\textbf{Case (a)} If $i_j=0$ for $j=1, 2, 3$, then $I_{\tilde{G}_2}(\chi)=\tilde{G}_2$. 

\textbf{Case (b)} Suppose  $i_1$ is non-zero. Taking $s=l=0,k=1$, we have $\xi^{i_1m}=1$ implies $m=0$. Taking $s=k=0,l=1$, we have $\xi^{-i_1n}=1$ implies $n=0$.  So $I_{\tilde{G}_2}(\chi)=N$.

\textbf{Case (c) :} Suppose $i_1 = 0$ and  $i_3\neq 0$. Then $I_{\tilde{G}_2}(\chi)=\langle N,  \alpha_1\alpha_2^{-{i_2}{i_3}^{-1}}\rangle$ is of order $p^7$.

\textbf{Case (d) :} Suppose $i_1 = 0$ and  $i_2 \neq 0, i_3=0$. Then $I_{\tilde{G}_2}(\chi)=\langle N, \alpha_2\rangle$ is of order $p^7$.
\\

Similarly as described above, we have 
$$\mathbb C \tilde{G}_2 \cong  \oplus_{i=1}^{p^4} \mathbb C \oplus_{i=1}^{(p^5-p^2)} \mathbb C^{p\times p} \oplus_{i=1}^{(p^4-p^3)} \mathbb C^{p^2\times p^2}.$$
Define  an isomorphism $\sigma: \Hom(A_1, \mathbb C^\times) \to \Hom(A_2, \mathbb C^\times)$ on the generators by
$$\sigma(\chi) (\alpha_2^{p})=\chi(\beta^p), \sigma(\chi)(\alpha_3)=\chi(\beta_1), \sigma(\chi)(\alpha_4)=\chi(\beta_2).$$  Then, there is a dimension preserving bijective correspondence between the sets 
$\irr(\tilde{G}_1\mid \chi)$ and $\irr(\tilde{G}_2\mid \sigma(\chi))$. Hence, the result follows from Lemma \ref{bijection_ord}.
\\

$(ix)$ By \cite[Theorem 1]{centraltype}, the groups $C_{p^2} \times C_{p^2}, (C_p)^4,  \Phi_2(22), $ $\Phi_2(1^4), \Phi_3(1^4)$  are all the  groups of central type of order $p^4$.
It follows from Theorem \ref{Schurp5} that, for all the groups $G$ listed in $(ix)$, $\Ho^2(G,\mathbb C^\times)$  are isomorphic. 
Taking $Z=Z(G) \cap G'$, by \cite[Table 3]{groupsp5} and Theorem \ref{capable}, we have the following exact sequence 
$$1 \to \Hom(Z, \mathbb C^\times)  \xrightarrow{\mathrm{tra}} \Ho^2(G/Z, \mathbb C^\times)  \xrightarrow{\mathrm{inf}} \Ho^2(G, \mathbb C^\times) \to 1.$$
Since $G/Z$ is not of central type, the result follows from Theorem \ref{centraltype}.
\\

$(xi)$ Now consider the groups $G_1=\Phi_3(2111)d$ and $G_2=\Phi_3(2111)e$. Then the following groups 
 of order $p^7$ 
\[
\tilde{G_1}=\langle \alpha,\alpha_i, \alpha_4,\gamma~(i=1,2,3) \mid [\alpha_i,\alpha]=\alpha_{i+1},[\alpha_1, \alpha_2]=\gamma,\gamma^p=\alpha^{p^2}=\alpha_i^p=\alpha_4^p=1\rangle\]
and 
\[
\tilde{G_2}=\langle \alpha,\alpha_i, \alpha_4,\gamma~(i=1,2,3) \mid [\alpha_i,\alpha]=\alpha_{i+1},[\alpha_1, \alpha_2]=\gamma,\gamma^p=\alpha^p=\alpha_1^{p^2}=\alpha_{i+1}^p=1\rangle
\]
are representation groups of $G_1$ and $G_2$ respectively. Now proceeding on the same lines as proof of $(vi)$,
we have 
$$\mathbb C \tilde{G}_i \cong  \oplus_{i=1}^{p^3} \mathbb C \oplus_{i=1}^{(2p^4-p^3-p)} \mathbb C^{p\times p} \oplus_{i=1}^{(p^3-2p^2+p)} \mathbb C^{p^2\times p^2}.$$
\\

 $(xiv)$ Proof  follows from Theorem \ref{extra_special}.
\\

$(xv)$ Proof follows similarly as the proof of $(ix)$.
\\

$(xvi)$ Let $G= \Phi_9(2111)a$ and $H= \Phi_9(2111)b_r$. As $p\geq 5$, 
We have  
$$\Phi _9(2111)a=\langle \alpha, \alpha_i, 1 \leq i \leq 4\mid [\alpha_j, \alpha]= \alpha_{j+1}, \alpha^p= \alpha_4, \alpha{_1}^{p}= \alpha_{j+1}^{p}=1, j=1,2,3 \rangle,$$
$$\Phi_9(2111)b_r=\langle \alpha^{\prime}, \alpha^{\prime}_i,1\leq i \leq 4 \mid [\alpha^{\prime}_j, \alpha^{\prime}]= \alpha^{\prime}_{j+1}, {\alpha^{\prime}}_1^{(p)}= {\alpha^{\prime}}_4^k, {\alpha^{\prime}}^p= {\alpha^{\prime}}_{j+1}^{(p)}=1,j=1,2,3 \rangle.$$
Here $Z(G)=\langle \alpha_4 \rangle$ and $Z(H)=\langle \alpha'_4\rangle$ and
  $$G/Z(G)= \Phi_3{(1^4)}=\langle \bar{\alpha}, \bar{\alpha}_i, 1 \leq i \leq 3\mid [\bar{\alpha}_j, \bar{\alpha}]= \bar{\alpha}_{j+1}, \bar{\alpha}^p= \bar{\alpha}{_j}^{p}= \bar{\alpha}_{3}^{p}=1, j=1,2\rangle.$$
We also have $H/Z(H)= \Phi_3(1^4)$.
By Theorem \ref{Schurp5}, the Schur multipliers of $G$ and $H$ are isomorphic to $C_p$.
It follows from \cite[Table 3]{groupsp5} and Theorem \ref{capable} that the rows are exact in the following diagram.
\[
\xymatrix{
	1 \ar[r] &  \mathrm{Hom}(Z(G), \mathbb C^\times) \ar[d] ^{\mathrm{\bar{i}}}\ar[r] ^{\mathrm{tra}_1} & \Ho^2(G/Z(G), \mathbb C^\times) \ar[d]^{\mathrm{\bar{\phi}}}  \ar[r]\ar[r]^{\mathrm{\operatorname{inf}}_{1}} & \Ho^2(G, \mathbb C^\times) \ar[r] &  1 ,\\
	1 \ar[r] &  \mathrm{Hom}(Z(H), \mathbb C^\times) \ar[r] ^{\mathrm{tra}_2} & \Ho^2(H/Z(H), \mathbb C^\times) \ar[r]\ar[r]^{\mathrm{\operatorname{inf}}_{2}} & \Ho^2(H, \mathbb C^\times) \ar[r] &  1 .
}
\]
Here $\mathrm{tra}_1:\mathrm{Hom}(Z(G), \mathbb C^\times ) \rightarrow \Ho^2(G/Z(G), \mathbb C^\times)$ is defined as follows: consider a section $\mu:  G/Z(G) \to G$ defined by $\mu(\bar{\alpha}^{i_1} \bar{\alpha}_1^{ i_2} \bar{\alpha}_2^{i_3}\bar{\alpha}_3^{i_4})=\alpha^{i_1} \alpha_1^{ i_2} \alpha_2^{i_3}\alpha_3^{i_4}$.
	Then for $X=\bar{\alpha}^{i_1} \bar{\alpha}_1^{ i_2} \bar{\alpha}_2^{i_3}\bar{\alpha}_3^{i_4}$ and $Y=\bar{\alpha}^{i'_1} \bar{\alpha}_1^{ i'_2} \bar{\alpha}_2^{i'_3}\bar{\alpha}_3^{i'_4}$ in $G/Z(G)$,
we have
\begin{flalign}
XY &=(\bar{\alpha}^{i_1} \bar{\alpha}_1^{ i_2} \bar{\alpha}_2^{i_3}\bar{\alpha}_3^{i_4})(\bar{\alpha}^{i_1^\prime} \bar{\alpha}_1^{i_2^\prime} \bar{\alpha}_2^{i_3^\prime}\bar{\alpha}_3^{i_4^\prime})\nonumber\\
&= \bar{\alpha}^{i_1+i_1^\prime}  \bar{\alpha}_1^{i_2+i_2^\prime} \bar{\alpha}_2^{i_1^\prime i_2+i_3+i_3^\prime}\bar{\alpha}_3^{i_2{i_1^\prime \choose 2}+{i_1^\prime i_3}+i_4+i_4^\prime}\nonumber.
\end{flalign}
Now  in $G$, we have,
\begin{flalign}
\mu(X)\mu(Y)& =(\alpha^{i_1} \alpha_1^{ i_2} \alpha_2^{i_3}\alpha_3^{i_4})(\alpha^{i_1^\prime} \alpha_1^{i_2^\prime} \alpha_2^{i_3^\prime}\alpha_3^{i_4^\prime})\nonumber\\
&=\alpha^{i_1+i_1^\prime}  \alpha_1^{i_2+i_2^\prime} \alpha_2^{i_1^\prime i_2+i_3+i_3^\prime}\alpha_3^{i_2{i_1^\prime \choose 2}+{i_1^\prime i_3}+i_4+i_4^\prime}\alpha_4^{i_1^\prime i_4+i_3{i_1^\prime \choose 2}+i_2{i_1^\prime \choose 3}}\nonumber.
\end{flalign}
Therefore, for $\chi \in \mathrm{Hom}(Z(G), \mathbb C^\times)$, we have $$\mathrm{tra_1(\chi)}(X,Y)=\chi\big(\mu(X)\mu(Y)\mu(XY)^{-1}\big)=\chi(\alpha_4)^{i_1^\prime i_4+i_3{i_1^\prime \choose 2}+i_2{i_1^\prime \choose 3}}.$$

Now consider the group $H$. Then
 $H/Z(H)= \Phi_3{(1^4)}=\langle {\bar{\alpha}}^{\prime}, {\bar{\alpha}}^{\prime}_i, 1 \leq i \leq 3\mid [{\bar{\alpha}}^{\prime}_j, {\bar{\alpha}}^{\prime}]= {\bar{\alpha}}^{\prime}_{j+1}, {\bar{\alpha}}^{\prime p}= \bar{\alpha}{_i}^{\prime p}= \bar{\alpha}_{3}^{\prime p}=1, j=1,2\rangle.$ 
One can check that   $\mathrm{tra}_2:\mathrm{Hom}(Z(H), \mathbb C^\times ) \rightarrow \Ho^2(H/Z(H), \mathbb C^\times)$ is also defined in the same manner as above. Hence,  for \\
$X=(\vt)^{i_1}(\vt_1)^{i_2} (\vt_2)^{i_3}(\vt_3)^{i_4}, 
Y=(\vt)^{i'_1} (\vt_1)^{ i'_2} (\vt_2)^{i'_3}(\vt_3)^{i'_4}\in  H/Z(H)$
we get 
\[
\mathrm{tra_2(\chi)}(X,Y)=\chi(\alpha '_4)^{i_1^\prime i_4+i_3{i_1^\prime \choose 2}+i_2{i_1^\prime \choose 3}} .
\]
Consider the isomorphisms $i: Z(H) \to Z(G)$ and $\phi:  H/Z(H) \to  G/Z(G)$ defined on the generator by $i(\bar{\alpha}_4)=\alpha_4$ and 

$$\phi((\vt)^{i_1}(\vt_1)^{i_2} (\vt_2)^{i_3}(\vt_3)^{i_4} )=\bar{\alpha}^{i_1} \bar{\alpha}_1^{i_2} \bar{\alpha}_2^{i_3}\bar{\alpha}_3^{i_4}$$ 
respectively.
Then, it is easy to check that the induced isomorphisms $\bar{i}$ and $\bar{\phi}$ will make the above diagram commutative, i.e., $\bar{\phi}\circ \tra_1=\tra_2\circ \bar{i}$.
Hence, by Corollary \ref{result1}, $G \sim_\mathbb C H$.

Using the same argument as above, we have  $\Phi_6(2111)a \sim_\mathbb C   \Phi_6(2111)b_r$.\\

Now it is enough to prove that $\Phi_6(2111)a \sim_\mathbb C \Phi_9(2111)a$.
Let $G_1=\Phi_6(2111)a $ and $G_2=\Phi_9(2111)a $. Then
\begin{eqnarray*}
    G_1=\Phi_6(2111)a=&&\langle \alpha_1, \alpha_2,\beta,\beta_1,\beta_2 \mid [\alpha_1,\alpha_2]=\beta,[\beta,\alpha_i]=\beta_i,\alpha_1^p=\beta_1,\\
    && \alpha_2^p=\beta^p=\beta_i^p=1 ~(i=1,2)\rangle,
\end{eqnarray*} 
and 
\begin{eqnarray*}
\tilde{G}_1 = && \langle\alpha_1, \alpha_2,\beta,\beta_i, 1\leq i \leq 3 \mid [\alpha_1,\alpha_2]=\beta,[\beta,\alpha_j]=\beta_j, [\beta_2,\alpha_2]=\beta_3, \alpha_1^p=\beta_1, \\
&& \alpha_2^p=\beta^p=\beta_i^p=1 ~(j=1,2)\rangle
\end{eqnarray*}
is a representation group of $G_1$ of order $p^6$ and $\tilde{G}_1/\langle \beta_3 \rangle \cong G_1 $. Now consider the abelian normal subgroup $N=\langle \beta, \beta_i, 1\leq i \leq 3 \rangle$ of $\tilde{G}_1$ of order $p^4$. Take $A_1=\langle \beta_3 \rangle.$
Let $\chi' \in \mathrm{Irr}(\langle \beta_3 \rangle)$ and  $\chi \in \mathrm{Irr}(N)$ such that $\chi |_{\langle \beta_3 \rangle }={\chi'}$. Let $\chi(\beta)=\xi^{i_1}, \chi(\beta_1)=\xi^{i_2}, \chi(\beta_2)=\xi^{i_3}, \chi(\beta_3)=\xi^{i_4}$, where $\xi$ is a primitive $p$-th root of unity and $0\leq i_1,i_2,i_3, i_4\leq (p-1)$.
Now any element $g$ of $\tilde{G}_1$ can be written as $g=\alpha_1^m\alpha_2^n n'$, for some $n'\in N$ and any element $n_1 \in N$ is of the form $n_1=\beta^x\beta_2^z h $, for some $h \in Z(\tilde{G}_1)$. Therefore,
we have
\begin{flalign}
	\chi^{\alpha_1^m\alpha_2^n}(\beta^x\beta_2^zh )=\chi(\beta^x\beta_2^zh ) 
	& \iff \chi(\alpha_2^{-n}\alpha_1^{-m}\beta^x\beta_2^z \alpha_1^m\alpha_2^n)=\chi(\beta^x\beta_2^z ) \nonumber\\
	& \iff \chi(\beta_1^{mx}\beta_2^{nx}\beta_3^{nz+x {n \choose 2}})=1,\nonumber \\
 & \iff \xi^{i_2{mx}+i_3{nx}+i_4{(nz+x {n \choose 2}})}=1.\nonumber
\end{flalign}
Thus $\alpha^m\alpha_1^nn' \in I_{\tilde{G}_1}(\chi)$ if and only if $\xi^{i_2{mx}+i_3{nx}+i_4{(nz+x {n \choose 2}})}=1.$
Now the following cases occur.

\textbf{Case (a)} Suppose $i_4=0$. If $i_2=i_3=0$, then $I_{\tilde{G_1}}(\chi)=\tilde{G_1}$, otherwise $I_{\tilde{G_1}}(\chi)$ is of order $p^5$.
Hence, it is easy to check that there are $p^2$ many one dimensional and $(p^3-1)$ many $p$-dimensional irreducible  representations.

\textbf{Case (b)} Suppose $i_4\neq 0$. If $i_2=0$, then $I_{\tilde{G_1}}(\chi)= \langle N, \alpha_1 \rangle$ is of order $p^5$. If $i_2\neq 0$, then $I_{\tilde{G_1}}(\chi)= N$. Hence, there are  $(p^3-p^2)$ many $p$-dimensional irreducible representations and $(p-1)^2$ many $p^2$-dimensional irreducible representations.

Thus, we have 
$$\mathbb C \tilde{G}_1 \cong \oplus_{i=1}^{p^2} \mathbb C \oplus_{i=1}^{2p^3-p^2-1} \mathbb C^{p\times p}\oplus_{i=1}^{p^2-2p+1} \mathbb C^{p^2\times p^2}.$$
\\

Next consider the group 
\begin{eqnarray*}
G_2=\Phi _9(2111)a=\langle \alpha, \alpha_i, 1 \leq i \leq 4\mid [\alpha_j, \alpha]= \alpha_{j+1}, \alpha^p= \alpha_4, \alpha{_1}^{(p)}= \alpha_{j+1}^{(p)}=1, j=1,2,3 \rangle.
\end{eqnarray*}
Then
\begin{eqnarray*}
\tilde{G}_2 = && \langle\alpha_1, \alpha_2,\beta,\beta_i, 1\leq i \leq 3 \mid [\alpha_1,\alpha_2]=\beta,[\beta,\alpha_j]=\beta_j, [\beta_2,\alpha_2]=\beta_3, \alpha_2^p=\beta_3, \\
&& \alpha_1^p=\beta^p=\beta_i^p=1 ~(j=1,2)\rangle
\end{eqnarray*}
is a representation group of $G_2$ of order $p^6$ and $\tilde{G}_2/\langle \beta_1 \rangle \cong G_2 $. Now consider the abelian normal subgroup $N=\langle \beta, \beta_i, 1\leq i \leq 3 \rangle$ of $\tilde{G}_2$ of order $p^4$. Take $A_2=\langle \beta_1 \rangle.$
Let $\chi' \in \mathrm{Irr}(\langle \beta_1 \rangle)$ and  $\chi \in \mathrm{Irr}(N)$ such that $\chi |_{\langle \beta_1 \rangle }={\chi'}$. Let $\chi(\beta)=\xi^{i_1}, \chi(\beta_1)=\xi^{i_2}, \chi(\beta_2)=\xi^{i_3}, \chi(\beta_3)=\xi^{i_4}$, where $\xi$ is a primitive $p^{th}$ root of unity and $0\leq i_1,i_2,i_3, i_4\leq (p-1)$.
Now 
\begin{flalign}
	\chi^{\alpha_1^m\alpha_2^n}(\beta^x\beta_2^zh )=\chi(\beta^x\beta_2^zh ) 
	& \iff \chi(\alpha_2^{-n}\alpha_1^{-m}\beta^x\beta_2^z \alpha_1^m\alpha_2^n)=\chi(\beta^x\beta_2^z ) \nonumber\\
	& \iff \chi(\beta_1^{mx}\beta_2^{nx}\beta_3^{nz+x {n \choose 2}})=1,\nonumber \\
 & \iff \xi^{i_2{mx}+i_3{nx}+i_4{(nz+x {n \choose 2}})}=1\nonumber
\end{flalign}
Now the following cases occur.

\textbf{Case (a)} Suppose $i_2=0$. If $i_3=i_4=0$, then $I_{\tilde{G_2}}(\chi)=\tilde{G_2}$, otherwise $I_{\tilde{G_2}}(\chi)=\langle N, \alpha_1\rangle$ is of order $p^5$.
Hence, it is easy to check that there are $p^2$ many one dimensional and $(p^3-1)$ many $p$-dimensional irreducible representations.

\textbf{Case (b)} Suppose $i_2\neq 0$. If $i_4=0$, then $I_{\tilde{G_2}}(\chi)$ is of order $p^5$. If $i_4\neq 0$, then $I_{\tilde{G_2}}(\chi)= N$. Hence, there are  $(p^3-p^2)$ many $p$-dimensional representations and $(p-1)^2$ many $p^2$-dimensional irreducible  representations.

Thus, we have 
$$\mathbb C \tilde{G}_2 \cong \oplus_{i=1}^{p^2} \mathbb C \oplus_{i=1}^{2p^3-p^2-1} \mathbb C^{p\times p}\oplus_{i=1}^{p^2-2p+1} \mathbb C^{p^2\times p^2}.$$
By Theorem \ref{proj-ord}, for each non-trivial cocycle $\alpha$ of $G_i~(i=1,2)$
$$\mathbb C^{\alpha} \tilde{G}_i \cong  \oplus_{i=1}^{p^2} \mathbb C^{p\times p}\oplus_{i=1}^{p-1} \mathbb C^{p^2\times p^2}.$$
Define  an isomorphism $\sigma: \Hom(A_1, \mathbb C^\times) \to \Hom(A_2, \mathbb C^\times)$ on the generators by
$$\sigma(\chi) (\beta_1)=\chi(\beta_3).$$  Then, there is a dimension preserving bijective correspondence between the sets 
$\irr(\tilde{G}_1\mid \chi)$ and $\irr(\tilde{G}_2\mid \sigma(\chi))$. Hence, the result follows from Lemma \ref{bijection_ord}.
\\
\\

\noindent $(xvii)$  Observe that, the groups $ \tilde{G}_1=\Phi_{(42,1)}$ and $ \tilde{G}_2=G_{(43,2r)} $ given in \cite{newman} are representation groups of $G_1=\Phi_6(221)b_{\frac{1}{2}(p-1)}$ and $G_2=\Phi_6(221)d_0$ respectively. In \cite[Pg. 2]{minimal}, it is mentioned that the  details given in \cite[Table 4.1]{James} are correct. Using those details, it follows that, for $i=1,2$,  
\begin{eqnarray*}
   \mathbb C \tilde{G}_i & \cong & \oplus_{i=1}^{p^2} \mathbb C \oplus_{i=1}^{(p^3-1)} \mathbb C^{p\times p} \oplus_{i=1}^{(p^2-p)} \mathbb C^{p^2\times p^2} \\
   & \cong &  \mathbb C {G_i} \oplus_{i=1}^{(p^2-p)} \mathbb C^{p^2\times p^2}
\end{eqnarray*}
Hence,  for each non-trivial cocycle $\alpha$ of $G_i$, we have
$$\mathbb C^{\alpha} G_i \cong  \oplus_{i=1}^{p} \mathbb C^{p^2\times p^2}.$$ 
Therefore, $G_1 \sim_\mathbb C G_2$.
\\

Proof of $(iv),  (x),  (xviii), (xxii)$ follows similarly as the proof of $\Phi_9(2111)a \sim_\mathbb C \Phi_9(2111)b_r$ in $(xvi)$.

\end{proof} 
\begin{remark}

We face difficulty finding  representation groups of the groups belonging to the set $S$. 
It follows from \cite[Table 4.1]{James} that for the groups $ G \in \Phi_{25}, H \in  \Phi_{26}$, and $K \in \Phi_{34}$, we have  $G/Z(G) \cong \Phi_3(221)b_1,  H/Z(H) \cong\Phi_3(221)b_\nu, K/Z(K) \cong \Phi_4(221)b$, and each of these groups $G, H$ and $K$ has  ordinary irreducible representations of degree $p^2$. Hence, by \cite[Theorem 3.2]{PS}, each of  $\Phi_3(221)b_r$ and $\Phi_4(221)b$ has some  irreducible projective representations of degree $p^2$. Therefore, by Theorem  \ref{Schurp5}, we conclude that the groups $\Phi_3(221)b_r,\Phi_4(221)b$ can not be in the list $(i)$ of Theorem \ref{groupsorderp5}. Also, we use Theorem  \ref{Schurp5} and  the isomorphism of  complex group algebras to propose the following open question.
\end{remark}

\noindent \textbf{Open question:}
The non-abelian groups of order $p^5$, where $p \geq 5$ is a prime,  consist of the following equivalence classes w.r.t the relation $\sim_\mathbb C$. \\
(a) All the equivalences classes as in Theorem \ref{groupsorderp5} except (xi) and (xxi), \\
(b) equivalence class (xi) changes to $\{\Phi_3(2111)d, \Phi_3(2111)e, \Phi_3(221)b_r,\Phi_4(221)b\},$\\
 (c) equivalence class (xxi) changes to $\{\Phi_{6}(1^5), \Phi_{9}(1^5)\}$, \\
 (d)   $\{ \Phi_4(221)d_{\frac{1}{2}(p-1)}, \Phi_4(221)f_0\}$.

\begin{remark}
  By \cite[Lemma 4.2]{MS},  for two groups $G,H$ of order $p^4$ if the following conditions are satisfied,\\
(i) $\mathbb CG \cong \mathbb CH$,\\
(ii) $\Ho^2(G, \mathbb C^\times) \cong \Ho^2(H, \mathbb C^\times),$\\
(iii) $G, H$ are not of central type, \\
  then $G \sim_\mathbb C H$.
Observe that this statement is not true for the groups of order $p^5$ by taking the groups $G=\Phi_6(2111)a, H=\Phi_6(221)d_0$.
\end{remark}


%
%
%
%
%
%
%
%
%
%

\section*{Acknowledgements} 
\noindent The authors are thankful to the editor for providing valuable suggestions which has greatly
improved the presentation of the paper.
We thank Prof. Eamonn O'Brien for pointing out a representation group of the group $\Phi_2(32)a_2$.  We are also grateful to the anonymous referee for carefully and critically reading the manuscript and providing many valuable comments. The second named author acknowledges the research support of the  Department of Science and Technology (INSPIRE Faculty Ref No.\linebreak DST/INSPIRE/04/2023/001200), Govt. of India. First and third named authors thank the National Institute of Science Education and Research, Bhubaneswar for providing an
excellent research facility.

\bibliographystyle{amsplain}
\bibliography{twisted}
\end{document}